\documentclass{birkjour}
\usepackage{amsmath, amssymb, amsthm}
\usepackage{graphicx}
\usepackage{subcaption}
\usepackage{booktabs}
\usepackage{geometry}
\usepackage{listings}
\usepackage{xcolor}
\usepackage{hyperref}
\usepackage[utf8]{inputenc} 
\usepackage{array}
\usepackage{float}
\usepackage{caption}
\usepackage[mathscr]{eucal}

 \theoremstyle{definition}
 
 \theoremstyle{remark}

 \numberwithin{equation}{section}
 \usepackage{url}
\newtheorem{definition}{Definition}[section]
\newtheorem{corollary}{Corollary}[section]
\newtheorem{theorem}{Theorem}[section]
\newtheorem{lemma}{Lemma}[section]

\newcounter{cases}
\newcounter{subcases}[cases]

\newcommand{\D}{\mathbb{D}}

\newenvironment{mycases}
  {%
    \setcounter{cases}{0}%
    \setcounter{subcases}{0}%
    \def\case
      {%
        \par\noindent
        \refstepcounter{cases}%
        \textbf{Case \thecases.}
      }%
  }
  {%
    \par
  }
\renewcommand*\thecases{\arabic{cases}}

\DeclareMathOperator{\spec}{\sigma}
\DeclareMathOperator{\specess}{\sigma_{\mathrm{ess}}}
\definecolor{codegreen}{rgb}{0,0.6,0}
\definecolor{codegray}{rgb}{0.5,0.5,0.5}
\definecolor{codepurple}{rgb}{0.58,0,0.82}
\definecolor{backcolour}{rgb}{0.95,0.95,0.92}

\lstdefinestyle{mystyle}{
    backgroundcolor=\color{backcolour},   
    commentstyle=\color{codegreen},
    keywordstyle=\color{magenta},
    numberstyle=\tiny\color{codegray},
    stringstyle=\color{codepurple},
    basicstyle=\ttfamily\footnotesize,
    breakatwhitespace=false,         
    breaklines=true,                 
    captionpos=b,                    
    keepspaces=true,                 
    numbers=left,                    
    numbersep=5pt,                  
    showspaces=false,                
    showstringspaces=false,
    showtabs=false,                  
    tabsize=2
}

\begin{document}

%
%
%
%
%
%
%
%
%

\title[  Algebraic and Spectral Properties $\cdots$]
{ Algebraic and Spectral properties of  slant Toeplitz and slant little Hankel Operators on weighted Bergman Space}

\author[Oinam Nilbir Singh]{Oinam Nilbir Singh}

\address{%
Department of Mathematics, Manipur University, Canchipur,795003, India}

\email{nilbirkhuman@manipuruniv.ac.in}

\author{M. P.  Singh}
\address{Department of Mathematics, Manipur University, Canchipur,795003, India}
\email{mpremjitmu@gmail.com}

\author{Thokchom Sonamani Singh}
\address{Department of Mathematics, Manipur University, Canchipur,795003, India}
\email{thokchomsonamani1994@gmail.com}


\date{January 1, 2004}

\begin{abstract}
This paper studies the \(k^{th}-\)order slant Toeplitz and slant little Hankel operators on the weighted Bergman space \(\mathcal{A}_\alpha^2(\mathbb{D})\). These operators are constructed using a slant shift operator \(W_k\) composed with classical Toeplitz and Hankel operators, respectively. We derive their matrix representations and establish criteria for boundedness, compactness, and normality. Commutativity conditions are obtained, showing that two such operators commute if and only if their symbols are linearly dependent. Normality is characterized as slant Toeplitz operators are normal only for constant symbols, while slant little Hankel operators are normal if the symbol is analytic. Compactness is shown to occur precisely when the symbol vanishes. Spectral properties, including essential spectra and eigenvalue distributions, are analyzed. Numerical simulations corroborate the theoretical findings and highlight key structural and computational differences between the two operator classes.
\end{abstract}

\maketitle

\section{Introduction}
Let $ \mathbb{D} = \{ z \in \mathbb{C}: \lvert z \rvert < 1 \}$ be the open unit disc and $dA$ be the normalized area measure on $ \mathbb{D}$. Let $dA_{\alpha}(z) = (1 + \alpha) (1 - \lvert z \rvert^2)^{\alpha} dA(z), \alpha  >  -1 $ be the normalized weighted Lebesgue area measure on the unit disk $ \mathbb{D}$. The weighted Bergman space $\mathcal{A}_{\alpha}^2(\mathbb{D})$ is the space of all analytic functions on $\mathbb{D}$ that are square integrable with respect to $dA_\alpha$. The weighted Bergman space $\mathcal{A}_{\alpha}^2(\mathbb{D})$ is a closed subspace of the weighted Hilbert space $\mathcal{L}^2_{\alpha}(\mathbb{D})$ that consists of the analytic functions on $\mathbb{D}$. So,  $\mathcal{A}_{\alpha}^2(\mathbb{D})$ inherits the inner product properties of $\mathcal{L}^2_{\alpha}(\mathbb{D})$. Let $f \in \mathcal{A}_{\alpha}^2(\mathbb{D}).$ The norm of $f$ is the usual norm $\mathcal{L}^2_{\alpha}(\mathbb{D})$\[  \displaystyle\lVert f \rVert_{\mathcal{L}^2_{\alpha}(\mathbb{D})} = \bigg\{ \int_\mathbb{D} \lvert f(z) \rvert^2 dA_{\alpha}(z)\bigg\}^{1/2}.\]
For any $f, g \in \mathcal{A}_{\alpha}^2(\mathbb{D})$ the inner product is
\[ \langle f,g \rangle = \int_\mathbb{D} f(z) \overline{g(z)} dA_{\alpha}(z) .\]
It is well known that $\bigg\{e_n(z) = \dfrac{z^n}{\gamma_n}, \forall n \in \mathbb{N} \cup \{0\} ~ {\rm and} ~ z \in \mathbb{D} \bigg\}$ is an orthonormal basis of $\mathcal{A}_{\alpha}^2(\mathbb{D})$ \cite{hed} and $\gamma_n = \sqrt{\dfrac{\Gamma(n+1) \Gamma(\alpha +2)}{\Gamma (n + \alpha +2 )}}$,  where $\Gamma(n)$ is the usual gamma function.   The weighted Bergman projection is an orthogonal projection $P_{\alpha}$ from $\mathcal{L}^2_{\alpha}(\mathbb{D})$ to $\mathcal{A}_{\alpha}^2(\mathbb{D})$ represented by \[P_\alpha f(z)= \langle f, K_z^\alpha \rangle= \int_\mathbb{D}K_\alpha(z,u) f(u) dA(u),\] where $K_\alpha(z,u) = K_u^\alpha(z) = \dfrac{1}{(1- \overline{u} z)^{2+\alpha}}$ is the unique reproducing kernel of $\mathcal{A}^2_{\alpha}(\mathbb{D})$.

The introduction of Hankel operators in the mid-$19^{th}$ century has played a pivotal role in the advancement of operator theory, owing to its profound implications and extensive applications in various mathematical problems such as rational approximation, interpolation, and prediction, among others \cite{ra,po,ke}. The study of Toeplitz and Hankel operators has garnered considerable attention from numerous mathematicians, leading to their exploration and generalization across various functional spaces.
In 1996, Ho \cite{ho} introduced the concept of a slant Toeplitz operator and conducted a comprehensive investigation into its fundamental properties. This development was further extended by Arora and his research collaborators in 2006 when they introduced the class of slant Hankel operators. Their work marked a significant step in understanding and generalizing operator theory.
Moreover, the algebraic properties of Toeplitz operators on the weighted Bergman space have recently been studied by Appuhamy and Joplin 
\cite{apun} in 2021, thereby contributing to the enrichment of the theoretical framework associated with these operators. Additionally, various aspects such as complex symmetry, boundedness, and hyponormality pertaining to Toeplitz operators, Hankel operators, slant Toeplitz operators, and slant Hankel operators have been rigorously investigated by several researchers\cite{gupta, kle, les, liu,zeng}. The continued interest in these classes of operators underscores their significance and the potential for further theoretical advancements in operator theory.

This paper presents a comprehensive study of the \(k^{th}\)order slant Toeplitz and slant little Hankel operators on the weighted Bergman space \(\mathcal{A}_\alpha^2(\mathbb{D})\), where \(\alpha > -1\) and \(k \in \mathbb{N}\). These operators are defined via composition with a slant-shift operator \(W_k\), which maps standard orthonormal basis elements to their scaled counterparts and introduces a discrete slant in the index structure. The slant Toeplitz operator is given by \(B_{k,\alpha}^\varphi = W_k T_\varphi^\alpha\), and the slant little Hankel operator by \(\mathcal{S}_{\varphi}^{k,\alpha} = W_k H_\varphi^\alpha\), where \(T_\varphi^\alpha\) and \(H_\varphi^\alpha\) denote the Toeplitz and Hankel operators with symbol \(\varphi\), respectively.

We derive detailed matrix representations of both operator classes and provide explicit formulas for their entries concerning the weighted Bergman basis. The boundedness and norm behavior of these operators are studied, with \(W_k\) shown to be a contraction. The paper establishes necessary and sufficient conditions for commutativity that two slant Toeplitz or slant little Hankel operators commute if and only if their harmonic or non-harmonic symbols are linearly dependent. 
The normality of the operators is also examined. It is shown that the slant Toeplitz operator is normal if and only if the symbol is constant, while the slant little Hankel operator is normal precisely when the symbol is analytic, i.e., has no anti-analytic component. The paper further investigates compactness properties. We prove that slant Toeplitz and slant little Hankel operators are compact if and only if their symbols vanish identically. 
The spectral behavior is studied in depth. Essential spectra of slant Toeplitz operators are shown to coincide with that of the underlying shift operator \(W_k\), typically the closed unit disc. In contrast, the essential spectrum of the slant little Hankel operator reduces to \(\{0\}\) under non-compactness. In particular, we show that \(\sigma_{\text{ess}}(B_{k,\alpha}^\varphi) = \sigma(W_k) = \overline{\mathbb{D}}\) when \(\varphi \not\equiv 0\), while \(\sigma_{\text{ess}}(\mathcal{S}_{\varphi}^{k,\alpha}) = \{0\}\) for analytic \(\varphi\).

To supplement the theoretical developments, we implement a numerical framework that visualizes matrix sparsity, decay of entries, commutator norms, and eigenvalue distributions for various analytic and anti-analytic symbols. The visual evidence confirms key theoretical predictions, such as faster decay and sparser matrices for slant little Hankel operators, as well as richer essential spectra for slant Toeplitz operators. Computational metrics further reveal that slant little Hankel operators are more efficient in terms of memory and time complexity.

This work unifies and extends classical results in operator theory by introducing and analyzing a broad class of slant-modified operators on weighted analytic function spaces. The algebraic, spectral, and computational insights provided here contribute meaningfully to the growing literature on structured operators in functional and harmonic analysis.

\section{Preliminaries}
In this section, we recall the classical definitions of Toeplitz and little Hankel operators on \(\mathcal{A}_\alpha^2(\mathbb{D})\), and derive their matrix entries with respect to the standard orthonormal basis. To define the slant analogues, we introduce the family of slant shift operators \(W_k\), indexed by \(k \geq 2\), and establish their boundedness and adjoint action. These operators allow us to define the \(k^{th}\)- order slant Toeplitz and slant little Hankel operators as compositions involving \(W_k\), the Bergman projection, and appropriate multiplication and conjugation operators. The structural properties and matrix formulations developed in this section form the analytical foundation for the spectral, algebraic, and compactness results to be investigated out in subsequent sections.

\begin{lemma} \cite{hw}
For any $ s, t \in \mathbb{N}$
\begin{align*}
P_\alpha(\overline{z}^t z^s)=\left\{\begin{array}{lll}
               \dfrac{\Gamma(s+1)\Gamma(s-t+\alpha +2)}{\Gamma(s+ \alpha + 2)\Gamma(s-t + 1)} z^{s-t} & \textnormal{if} & s \geq t \\
                0  & \textnormal{if}  &  s < t.
            \end{array}\right.
\end{align*}
\end{lemma}
\begin{definition}
  For $\phi \in L^\infty(\mathbb{D}, dA_\alpha)$, the Toeplitz operator $\mathcal{T}_\phi^\alpha: \mathcal{A}^2_\alpha(\mathbb{D}) \to \mathcal{A}^2_\alpha(\mathbb{D})$ is defined as
$\mathcal{T}_\phi^\alpha f = P_\alpha (\phi \cdot f),$
where $P_\alpha$ is the orthogonal projection from $\mathcal{L}^2_\alpha(\mathbb{D})$ onto $\mathcal{A}^2_\alpha(\mathbb{D})$.  
\end{definition}
The $(m,n)^{\text{th}}$ entry of the matrix representation of the Toeplitz operator $\mathcal{T}_{\phi}^{\alpha}$ on the weighted Bergman space $\mathcal{A}_{\alpha}^2(\mathbb{D})$ with respect to the orthonormal basis is given by 
\begin{align*}
\langle \mathcal{T}_\phi^{\alpha} e_n, e_m \rangle &= \frac{1}{\gamma_n \gamma_m} \langle \phi z^n, z^m \rangle \\
&= \frac{1}{\gamma_n \gamma_m} \int_{\mathbb{D}} \phi(z) z^n \overline{z}^m  dA_\alpha(z).
\end{align*}
For a symbol $\phi(z) = \displaystyle\sum_{j=0}^{\infty} a_j \bar{z}^j + \sum_{j=1}^{\infty} b_j z^j$, 
\[
\langle \mathcal{T}_\phi^{\alpha} e_n, e_m \rangle = 
\begin{cases} 
\dfrac{\gamma_n}{\gamma_m} a_{n-m} & n \geq m \\[2mm]
\dfrac{\gamma_m}{\gamma_n} b_{m-n} & n < m
\end{cases}
\]
For \(\alpha = 1\), the Toeplitz operator \(\mathcal{T}_\phi^1\) on the weighted Bergman space has the matrix representation
\[
\mathcal{T}_\phi^1 = 
\begin{bmatrix}
a_0 & 
\dfrac{a_1}{\sqrt{3}} & \dfrac{a_2}{\sqrt{6}} & \dfrac{a_3}{\sqrt{10}} & \cdots  \\[4mm]
\dfrac{b_1}{\sqrt{3}} & a_0 & \dfrac{a_1}{\sqrt{2}} & \sqrt{\dfrac{3}{10}}  a_2 & \cdots  \\[4mm]
\dfrac{b_2}{\sqrt{6}} & \dfrac{b_1}{\sqrt{2}} & a_0 & \sqrt{\dfrac{3}{5}}  a_1 & \cdots  \\[4mm]
\dfrac{b_3}{\sqrt{10}} & \sqrt{\dfrac{3}{10}}  b_2 & \sqrt{\dfrac{3}{5}}  b_1 & a_0 & \cdots \\[4mm]
\vdots & \vdots & \vdots & \vdots & \ddots
\end{bmatrix}
\]
\begin{definition}\label{mydefn1} {\rm\cite{gupta}}
Let $ \phi\,\in\mathcal{L}^\infty (\mathbb{D},dA)$ and for any $f \in \mathcal{A}_{\alpha}^2(\mathbb{D}) $ the little Hankel operator on the weighted Bergman space, $\mathcal{{H}}_{\phi}^{\alpha}:  \mathcal{A}_{\alpha}^2(\mathbb{D}) \rightarrow  \mathcal{A}_{\alpha}^2(\mathbb{D})$ is defined as $\mathcal{{H}}_{\phi}^{\alpha} = P_{\alpha}J M_{\phi}$ where $P_{\alpha}$ is the orthogonal weighted Bergman projection and $ J $ is the  flip operator on $ \mathcal{L}^2(\mathbb{D},dA)$.
\end{definition}
For a symbol $\phi(z) = \displaystyle\sum_{j=0}^\infty a_j \bar{z}^j + \sum_{j=1}^\infty b_j z^j \in \mathcal{L}^{\infty}(\mathbb{D}, dA)$, 
the $(m,n)^{\text{th}}$ entry of the matrix representation of the little Hankel operator $\mathcal{H}_{\phi}^{\alpha}$ 
on the weighted Bergman space $\mathcal{A}_{\alpha}^2(\mathbb{D})$ with respect to the orthonormal basis 
$e_n(z) = \dfrac{z^n}{\gamma_n}$ for $n \in \mathbb{N} \cup \{0\}$ is given by
\begin{align*}
\langle \mathcal{H}_{\phi}^{\alpha} e_n, e_m \rangle 
&= \frac{1}{\gamma_n \gamma_m} \langle P_{\alpha} J\phi(z) z^n, z^m \rangle \\
&= \frac{1}{\gamma_n \gamma_m} \left\langle \left( \sum_{j=0}^\infty a_j \bar{z}^j + \sum_{j=1}^\infty b_j z^j \right) z^n, \bar{z}^m \right\rangle \\
&= \frac{1}{\gamma_n \gamma_m} \left\langle \sum_{j=0}^\infty a_j \bar{z}^j z^n, \bar{z}^m \right\rangle \\
&= \frac{1}{\gamma_n \gamma_m}  \frac{1}{\pi} \int_0^{2\pi} \int_0^1 \sum_{j=0}^\infty a_j r^{j+n+m} e^{i\theta(n+m-j)} (\alpha+1)(1-r^2)^\alpha r  dr  d\theta \\
&= \frac{\gamma_{n+m}^2}{\gamma_n \gamma_m} a_{n+m}.
\end{align*}
When $\alpha = 0$, we obtain the little Hankel operator on the standard Bergman space. 
For $\alpha = 1$, the matrix representation simplifies to
\[
\mathcal{H}_{\phi}^{1} = 
\begin{bmatrix}
a_0 & \dfrac{1}{\sqrt{6}} a_1 & \dfrac{1}{\sqrt{12}} a_2 & \dfrac{1}{\sqrt{20}} a_3 & \cdots \\[8pt]
\dfrac{1}{\sqrt{6}} a_1 & \dfrac{1}{6} a_2 & \dfrac{\sqrt{6}}{20} a_3 & \dfrac{1}{\sqrt{30}} a_4 & \cdots \\[8pt]
\dfrac{1}{\sqrt{12}} a_2 & \dfrac{\sqrt{6}}{20} a_3 & \dfrac{1}{5} a_4 & \dfrac{1}{\sqrt{42}} a_5 & \cdots \\[8pt]
\dfrac{1}{\sqrt{20}} a_3 & \dfrac{1}{\sqrt{30}} a_4 & \dfrac{1}{\sqrt{42}} a_5 & \dfrac{3}{14} a_6 & \cdots \\[8pt]
\vdots & \vdots & \vdots & \vdots & \ddots
\end{bmatrix}.
\]

\begin{definition}
For $k \geq 2$ and $m \geq 0$, define the operator $W_k: \mathcal{A}_{\alpha}^2(\mathbb{D}) \to \mathcal{A}_{\alpha}^2(\mathbb{D})$ by
\[
W_k e_m(z) = 
\begin{cases} 
\dfrac{z^{m/k}}{\gamma_{m/k}} & \text{if }  k \mid m, \\
0 & \text{otherwise}.
\end{cases}
\]
\end{definition}

\begin{lemma}
Let \( \{e_n(z)\}_{n=0}^\infty \) denote the orthonormal basis of the weighted Bergman space, where \( e_n(z) = \frac{z^n}{\gamma_n} \). Then the adjoint of the operator \( W_k \), denoted \( W_k^* \), acts on the orthonormal basis as
\[
W_k^* e_m(z) = \frac{\gamma_m}{\gamma_{km}} e_{km}(z).
\]
\end{lemma}
\begin{lemma}
$W_k$ is a bounded linear operator with $\| W_k \| = 1$.
\end{lemma}
\begin{theorem}
For any $f \in \mathcal{A}_{\alpha}^2(\mathbb{D})$, $\| W_k^* f \|_2 \leq \| f \|_2$.
\end{theorem}

\begin{definition}
For integer $k \geq 2$, the slant Toeplitz operator $\mathcal{B}_\phi^{k,\alpha} : \mathcal{A}_\alpha^2(\mathbb{D}) \to \mathcal{A}_\alpha^2(\mathbb{D})$ is defined as
\[
\mathcal{B}_\phi^{k,\alpha}(f) = W_k \mathcal{T}^\alpha_\phi(f) \; \forall  f \in  \mathcal{A}_{\alpha}^2(\mathbb{D}).
\]
\end{definition}
The matrix elements \( \langle \mathcal{B}^{k, \alpha}_\phi e_n, e_m \rangle \) are derived as follows
\[
\langle \mathcal{B}^{k, \alpha}_\phi e_n, e_m \rangle 
= \langle W_k \mathcal{T}^\alpha_\phi e_n, e_m \rangle 
= \langle \mathcal{T}^\alpha_\phi e_n, W_k^* e_m \rangle 
= \frac{\gamma_m}{\gamma_{k m}} \langle \mathcal{T}_\phi e_n, e_{k m} \rangle.
\]
Using the matrix elements of the Toeplitz operator \( \mathcal{T}^\alpha_\phi \), we have
\[
\langle \mathcal{T}^\alpha_\phi e_n, e_{k m} \rangle = 
\begin{cases}
\dfrac{\gamma_n}{\gamma_{km}} a_{n-km} & \text{if } n \geq km, \\[4mm]
\dfrac{\gamma_{km}}{\gamma_n} b_{km-n} & \text{if } n < km
\end{cases}
\]
and therefore,
\[
\langle \mathcal{B}^{k, \alpha}_\phi e_n, e_m \rangle = 
\begin{cases} 
\dfrac{\gamma_m \gamma_n}{\gamma_{km}^2} a_{n-km} & \text{if } n \geq km, \\[4mm]
\dfrac{\gamma_m}{\gamma_n} b_{km-n} & \text{if } n < km.
\end{cases}
\]
For $\alpha = 1$, the matrix representation simplifies to
\[
\mathcal{B}_{\phi}^{2,1} = 
\begin{bmatrix}
a_0 & 
\dfrac{\sqrt{6}}{3} a_1 & 
\dfrac{\sqrt{2}}{2} a_2 & 
\sqrt{\dfrac{2}{5}}  a_3& \cdots  \\[4mm]
\dfrac{2}{\sqrt{6}} b_2 & 
a_0 & 
\dfrac{\sqrt{6}}{3} a_1 & 
\dfrac{\sqrt{2}}{2} a_2& \cdots  \\[4mm]
\dfrac{2}{\sqrt{12}} b_4 & 
\dfrac{2}{\sqrt{6}} b_2 & 
a_0 & 
\dfrac{\sqrt{6}}{3} a_1 & \cdots \\[4mm]
\dfrac{2}{\sqrt{20}} b_6 & 
\dfrac{2}{\sqrt{10}} b_4 & 
\dfrac{2}{\sqrt{6}} b_2 & 
a_0 & \cdots  \\[4mm]
\vdots & \vdots & \vdots & \vdots & \ddots
\end{bmatrix}
\]
\begin{definition}\label{mydef5}
For $k \geq 2$ and $\phi(z) \in \mathcal{L}^{\infty}(\mathbb{D},dA)$, a $k^{th}$ order slant little Hankel  operators on the weighted Bergman space, $\mathcal{ S}_{\phi}^{k,\alpha}: \mathcal{A}_{\alpha}^2(\mathbb{D}) \rightarrow \mathcal{A}_{\alpha}^2(\mathbb{D}) $  is defined by \[ \mathcal{ S}_{\phi}^{k,\alpha}(f) = {W}_k \mathcal{ H}_{\phi}^{\alpha}(f) \; \forall  f \in  \mathcal{A}_{\alpha}^2(\mathbb{D}). \]
\end{definition}
The $(m,n)^{th}$ entry of the slant little Hankel operator $\mathcal{S}_{\phi}^{k,\alpha}$ on the weighted Bergman space is given by
\begin{align*}
\big\langle \mathcal{S}_{\phi}^{k,\alpha} e_n(z), {e_m(z)}\big\rangle 
= & \big\langle W_kP_{\alpha}J \phi(z) e_n(z),{e_m(z)} \big\rangle \cr 
= &  \dfrac{\gamma_m \gamma_{n+km}^2}{\gamma_n \gamma_{km}^2} a_{n+km}.
\end{align*}
For \(\alpha = 0\) and \(k = 2\), the operator reduces to the slant Hankel operator on the Bergman space, with its matrix representation coinciding with that obtained by A. Gupta and B. Gupta\cite{gupta}. In the present context, we extend this notion by further generalizing the slant Hankel operator on the weighted Bergman space. Specifically, the \((m,n)^{\text{th}}\) entry of the slant little Hankel operator on the weighted Bergman space for \(\alpha = 1\) and \(k = 2\) is given by  
\[
\mathcal{S}_\phi^{2,1} =
\frac{(2m+1)(2m+2) \cdot \sqrt{(n+1)(n+2)}}{(n+2m+1)(n+2m+2) \cdot \sqrt{(m+1)(m+2)}} \cdot a_{n+2m}
\]
and its corresponding matrix representation is provided below
\[
\mathcal{S}_{\phi}^{2, 1} = 
\begin{bmatrix}
 a_0 & a_1 & a_2 & a_3 & \cdots \\[10pt]
\sqrt{2} a_2 & \frac{3}{\sqrt{2}} a_3 & \frac{6}{\sqrt{10}} a_4 & \frac{10}{\sqrt{18}} a_5 & \cdots \\[10pt]
\frac{5}{\sqrt{2}} a_4 & \frac{15}{\sqrt{6}} a_5 & \frac{15\sqrt{3}}{\sqrt{21}} a_6 & \frac{5\sqrt{5}}{\sqrt{12}} a_7 & \cdots \\[10pt]
\frac{7\sqrt{2}}{\sqrt{35}} a_6 & \frac{7\sqrt{6}}{\sqrt{5}} a_7 & \frac{28\sqrt{3}}{15} a_8 & \frac{28\sqrt{5}}{5\sqrt{11}} a_9 & \cdots \\[10pt]
\vdots & \vdots & \vdots & \vdots & \ddots
\end{bmatrix}
\]

 The \((m,n)^{\text{th}}\) entry of the slant little Hankel operator on the weighted Bergman space for \(\alpha = 2\) and \(k = 2\) is given by  

\begin{align*}
 \mathcal{S}_{\phi}^{2,2}  =
&\left[ \frac{(2m+1)(2m+2)(2m+3)}{(n+2m+1)(n+2m+2)(n+2m+3)} \right] \\
&\cdot \left[ \frac{\sqrt{(n+1)(n+2)(n+3)}}{\sqrt{(m+1)(m+2)(m+3)}} \right] \cdot a_{n+2m}
\end{align*}
where \(m\) and \(n\) are non-negative integers. Its corresponding matrix representation is provided below
\[
\mathcal{S}_{\phi}^{2,2} = 
\begin{bmatrix}
 a_0 & a_1 & a_2 & a_3 & a_4 & \cdots \\[10pt]
\sqrt{6} a_2 & \frac{5\sqrt{6}}{\sqrt{3}} a_3 & \frac{5\sqrt{6}}{\sqrt{5}} a_4 & \frac{35}{\sqrt{21}} a_5 & \frac{105}{\sqrt{45}} a_6 & \cdots \\[10pt]
\frac{35\sqrt{6}}{\sqrt{10}} a_4 & \frac{70\sqrt{6}}{\sqrt{15}} a_5 & \frac{105\sqrt{6}}{\sqrt{21}} a_6 & \frac{210\sqrt{6}}{\sqrt{28}} a_7 & \frac{1260}{\sqrt{36}} a_8 & \cdots \\[10pt]
\frac{84\sqrt{6}}{\sqrt{15}} a_6 & \frac{252\sqrt{6}}{\sqrt{21}} a_7 & \frac{504\sqrt{6}}{\sqrt{28}} a_8 & \frac{792\sqrt{6}}{\sqrt{36}} a_9 & \frac{1188\sqrt{6}}{\sqrt{45}} a_{10} & \cdots \\[10pt]
\vdots & \vdots & \vdots & \vdots & \vdots & \ddots
\end{bmatrix}
\] 
\section{Commutativity }
The commutativity of operator classes such as Toeplitz and Hankel operators has long been a subject of significant interest in operator theory, particularly due to its connections with function theory and operator algebras. In the unweighted setting, the commutativity of slant Hankel operators with harmonic symbols was studied by A. Gupta and B. Gupta in \cite{gupta}, where necessary and sufficient conditions were obtained. In this section, we extend these results to the weighted Bergman space \(\mathcal{A}_\alpha^2(\mathbb{D})\), and investigate the commutativity of the \(k^{th}\)-order slant Toeplitz and slant little Hankel operators under both harmonic and non-harmonic symbols.
We first show that if the symbols are harmonic and linearly dependent, then the associated slant Toeplitz and slant little Hankel operators commute. We then prove the converse that commutativity of the operators implies linear dependence of the symbols, thereby establishing an equivalence. Finally, we consider a special class of non-harmonic symbols, and under suitable algebraic conditions, identify situations in which the corresponding slant little Hankel operators commute. These results deepen the understanding of the algebraic structure of slant-type operators on weighted analytic function spaces.

\begin{theorem}
\label{commst}
Let \( \varphi(z) = \displaystyle\sum_{i=-n}^{n} a_i \dfrac{z^i}{\gamma_i} \) and \( \psi(z) = \displaystyle\sum_{j=-n}^{n} b_j \dfrac{z^j}{\gamma_j} \) be bounded harmonic functions on the weighted Bergman space \( \mathcal{A}_\alpha^2(\mathbb{D}) \), where \( n \geq 0 \) is an integer, and \( a_n \neq 0 \), \( b_n \neq 0 \). Then the \(k^{\text{th}}\)-order slant Toeplitz operators \( \mathcal{B}_{\varphi}^{k,\alpha} \) and \( \mathcal{B}_{\psi}^{k,\alpha} \) commute if and only if \( \varphi \) and \( \psi \) are linearly dependent.
\end{theorem}
\begin{proof}
(\(\Leftarrow\)) Suppose \( \psi = c \varphi \) for some constant \( c \in \mathbb{C} \). Then \( \mathcal{B}_{\psi}^{k,\alpha} = c \mathcal{B}_{\varphi}^{k,\alpha} \), and hence
\[
\mathcal{B}_{\varphi}^{k,\alpha} \mathcal{B}_{\psi}^{k,\alpha} = c \mathcal{B}_{\varphi}^{k,\alpha} \mathcal{B}_{\varphi}^{k,\alpha} = \mathcal{B}_{\psi}^{k,\alpha} \mathcal{B}_{\varphi}^{k,\alpha}.
\]
(\(\Rightarrow\)) Conversely, suppose \( \mathcal{B}_{\varphi}^{k,\alpha} \mathcal{B}_{\psi}^{k,\alpha} = \mathcal{B}_{\psi}^{k,\alpha} \mathcal{B}_{\varphi}^{k,\alpha} \). Then their adjoints also commute
\begin{equation}\label{eq:adjoint_comm}
(\mathcal{B}_{\varphi}^{k,\alpha})^* (\mathcal{B}_{\psi}^{k,\alpha})^* = (\mathcal{B}_{\psi}^{k,\alpha})^* (\mathcal{B}_{\varphi}^{k,\alpha})^*.
\end{equation}
Recall that \( (\mathcal{B}_{\phi}^{k,\alpha})^* = \mathcal{T}^\alpha_{\overline{\phi}} W_k^* \), so equation~\eqref{eq:adjoint_comm} becomes
\[
\mathcal{T}^\alpha_{\overline{\varphi}} W_k^* \mathcal{T}^\alpha_{\overline{\psi}} W_k^* = \mathcal{T}^\alpha_{\overline{\psi}} W_k^* \mathcal{T}^\alpha_{\overline{\varphi}} W_k^*.
\]
Consider the orthonormal basis \( \{e_m(z) = z^m / \gamma_m\}_{m=0}^{\infty} \). For any \( m \geq 0 \), we have
\[
W_k^* e_m = \frac{\gamma_m}{\gamma_{km}} e_{km}.
\]
Applying both sides of equation~\eqref{eq:adjoint_comm} to \( e_m \)
\begin{align*}
(\mathcal{T}^\alpha_{\overline{\varphi}} W_k^* \mathcal{T}^\alpha_{\overline{\psi}} W_k^*) e_m 
&= \mathcal{T}^\alpha_{\overline{\varphi}} W_k^* \left( \frac{\gamma_m}{\gamma_{km}} \mathcal{T}_{\overline{\psi}} e_{km} \right) \\
&= \frac{\gamma_m}{\gamma_{km}} \mathcal{T}^\alpha_{\overline{\varphi}} W_k^* \left( \sum_{j=-n}^{n} \overline{b_j} \frac{\gamma_{km-j}}{\gamma_j \gamma_{km}} c_{km,j} e_{km-j}, \right)
\end{align*}
where \( c_{p,q} = \dfrac{\Gamma(p+1)\Gamma(p-q+\alpha+2)}{\Gamma(p+\alpha+2)\Gamma(p-q+1)} \).
This reduces to
\[
\frac{\gamma_m}{\gamma_{km}} \sum_{j=-n}^{n} \sum_{i=-n}^{n} \overline{a_i} \overline{b_j} \frac{\gamma_{km-j} \gamma_{(km-j)-i}}{\gamma_i \gamma_j \gamma_{km}} c_{km,j} c_{km-j,i} e_{(km-j)-i}.
\]
Similarly, computing \( \mathcal{T}^\alpha_{\overline{\psi}} W_k^* \mathcal{T}^\alpha_{\overline{\varphi}} W_k^* e_m \) gives
\[
\frac{\gamma_m}{\gamma_{km}} \sum_{i=-n}^{n} \sum_{j=-n}^{n} \overline{b_i} \overline{a_j} \frac{\gamma_{km-i} \gamma_{(km-i)-j}}{\gamma_j \gamma_i \gamma_{km}} c_{km,i} c_{km-i,j} e_{(km-i)-j}.
\]
Equating coefficients of \( e_{km-i-j} \), we obtain
\[
\overline{a_i} \overline{b_j} \frac{\gamma_{km-j} \gamma_{km-j-i}}{\gamma_i \gamma_j \gamma_{km}} c_{km,j} c_{km-j,i} = \overline{b_i} \overline{a_j} \frac{\gamma_{km-i} \gamma_{km-i-j}}{\gamma_j \gamma_i \gamma_{km}} c_{km,i} c_{km-i,j}.
\]
Using the asymptotic estimates \( \gamma_p \sim p^{-(\alpha+1)/2} \) and \( c_{p,q} \sim p^{-\alpha-1} \) as \( p \to \infty \), the ratios tend to unity
\[
\frac{\gamma_{km-j} \gamma_{km-j-i} c_{km,j} c_{km-j,i}}{\gamma_{km-i} \gamma_{km-i-j} c_{km,i} c_{km-i,j}} \sim 1.
\]
Thus, for sufficiently large \( m \), we get
\[
\overline{a_i} \overline{b_j} = \overline{b_i} \overline{a_j} \quad \text{for all } i,j \in \{-n, \dots, n\},
\]
which implies
\[
a_i b_j = b_i a_j \quad \text{for all } i,j.
\]
Since \( b_n \neq 0 \), define \( \lambda = a_n / b_n \). Then for each \( i \),
\[
a_i b_n = b_i a_n \Rightarrow a_i = \lambda b_i,
\]
and hence \( \varphi = \lambda \psi \). This proves the linear dependence of \( \varphi \) and \( \psi \).
\end{proof}

\begin{theorem}\label{commsh}
Let $\varphi(z) = \displaystyle\sum_{i=-n}^n a_i \dfrac{z^i}{\gamma_i}$ and $\psi(z) = \displaystyle\sum_{j=-n}^n b_j \dfrac{z^j}{\gamma_j}$, where $n \geq 0$ is an integer, $a_n \neq 0$, $b_n \neq 0$, and $\varphi, \psi \in \mathcal{L}^{\infty}(\mathbb{D}, dA_{\alpha})$. Then the operators $\mathcal{S}_{\varphi}^{k,\alpha}$ and $\mathcal{S}_{\psi}^{k,\alpha}$ commute if and only if $\varphi$ and $\psi$ are linearly dependent.
\end{theorem}

\begin{proof}
($\Leftarrow$) If $\psi = c\varphi$ for some $c \in \mathbb{C}$, then $\mathcal{S}_{\psi}^{k,\alpha} = c\mathcal{S}_{\varphi}^{k,\alpha}$, and commutativity follows immediately.\\
($\Rightarrow$) Assume $\mathcal{S}_{\varphi}^{k,\alpha} \mathcal{S}_{\psi}^{k,\alpha} = \mathcal{S}_{\psi}^{k,\alpha} \mathcal{S}_{\varphi}^{k,\alpha}$. Then their adjoints commute:
\begin{equation}\label{eq:adj-comm}
(\mathcal{S}_{\varphi}^{k,\alpha})^* (\mathcal{S}_{\psi}^{k,\alpha})^* = (\mathcal{S}_{\psi}^{k,\alpha})^* (\mathcal{S}_{\varphi}^{k,\alpha})^*.
\end{equation}
Consider the orthonormal basis \[\{e_m(z) = z^m / \gamma_m\}_{m=0}^\infty\] where $\gamma_m = \sqrt{\dfrac{\Gamma(m+1)\Gamma(\alpha+2)}{\Gamma(m+\alpha+2)}}$. 
\begin{mycases}
   \case  If $n=0$ then $\varphi = a_0/\gamma_0$, $\psi = b_0/\gamma_0$ are constants, hence linearly dependent.
\case If $n>0$ then for each $m \geq 0$, apply both sides of \eqref{eq:adj-comm} to $e_m$
\begin{align*}
&(\mathcal{S}_{\varphi}^{k,\alpha})^* (\mathcal{S}_{\psi}^{k,\alpha})^* e_m \\
&= \frac{\gamma_m}{\gamma_{km}} \overline{\varphi} J P_{\alpha} W_k^* \overline{\psi} J P_{\alpha} e_{km} \\
&= \frac{\gamma_m}{\gamma_{km}^2} \overline{\varphi} \sum_{j=-n}^n \overline{b_j} \frac{\overline{\gamma}_{j+km}}{\gamma_j} \overline{e_{j+km}} \\
&= \frac{\gamma_m}{\gamma_{km}^2} \sum_{i=-n}^n \sum_{j=-n}^n \overline{a_i} \overline{b_j} \frac{\overline{\gamma}_{j+km}^2}{\overline{\gamma}_{k(j+km)} \gamma_i \gamma_j} \overline{z}^i z^{k(j+km)}
\end{align*}
Similarly,
\begin{align*}
&(\mathcal{S}_{\psi}^{k,\alpha})^* (\mathcal{S}_{\varphi}^{k,\alpha})^* e_m \\
&= \frac{\gamma_m}{\gamma_{km}^2} \sum_{j=-n}^n \sum_{i=-n}^n \overline{b_j} \overline{a_i} \frac{\overline{\gamma}_{i+km}^2}{\overline{\gamma}_{k(i+km)} \gamma_j \gamma_i} \overline{z}^j z^{k(i+km)}
\end{align*}
By \eqref{eq:adj-comm}, these expressions are equal. Since the monomials $\{\overline{z}^p z^q\}$ are linearly independent, coefficients of corresponding terms must be equal. For each pair $(i,j)$ and for all $m \geq 0$, we have:
\begin{equation}\label{eq:coeff-eq}
\overline{a_i} \overline{b_j} \frac{\overline{\gamma}_{j+km}^2}{\overline{\gamma}_{k(j+km)} \gamma_i \gamma_j} 
= \overline{b_i} \overline{a_j} \frac{\overline{\gamma}_{i+km}^2}{\overline{\gamma}_{k(i+km)} \gamma_i \gamma_j}
\end{equation}
which simplifies to:
\begin{equation}\label{eq:simplified}
\overline{a_i} \overline{b_j} \cdot \frac{\gamma_{k(j+km)}}{\gamma_{j+km}^2} 
= \overline{b_i} \overline{a_j} \cdot \frac{\gamma_{k(i+km)}}{\gamma_{i+km}^2}
\end{equation}
To analyze the asymptotic behavior, recall that $\gamma_q = \sqrt{\dfrac{\Gamma(q+1)\Gamma(\alpha+2)}{\Gamma(q+\alpha+2)}}$. Using Stirling's approximation $\Gamma(y+1) \sim \sqrt{2\pi y} (y/e)^y$ as $y \to \infty$, we obtain for fixed $p \in \mathbb{Z}$
\[
\frac{\gamma_{k(p + km)}}{\gamma_{p + km}^2} \sim k^{\alpha+1} \quad \text{as} \quad m \to \infty.
\]
Fix $i,j \in \{-n,\dots,n\}$. Taking $m \to \infty$ in \eqref{eq:simplified} yields
\[
\overline{a_i} \overline{b_j} \cdot k^{\alpha+1} = \overline{b_i} \overline{a_j} \cdot k^{\alpha+1}
\]
Thus $\overline{a_i} \overline{b_j} = \overline{b_i} \overline{a_j}$ for all $i,j$, or equivalently
\[
a_i b_j = b_i a_j \quad \forall i,j \in \{-n,\dots,n\}.
\]
Since $b_n \neq 0$, set $\lambda = a_n / b_n$. Then for each $i$
\[
a_i b_n = b_i a_n \implies a_i = \lambda b_i
\]
Hence $\varphi = \lambda \psi$, so they are linearly dependent.
\end{mycases}
\end{proof}

\begin{theorem}
Let $\phi(z) = a_{m,n} \dfrac{z^m}{\gamma_m} \dfrac{\overline{z}^n}{\gamma_n}$ and 
$\psi(z) = b_{r,s} \dfrac{z^r}{\gamma_r} \dfrac{\overline{z}^s}{\gamma_s}$. If $s > r + kp$ and $n > m + kp$ for some integer $p \geq 0$, 
and $m - n = r - s$, then the operators commute
\[
\mathcal{S}_\phi^{k,\alpha} \mathcal{S}_\psi^{k,\alpha} = \mathcal{S}_\psi^{k,\alpha} \mathcal{S}_\phi^{k,\alpha}.
\]
\end{theorem}

\begin{proof}
We prove commutativity by showing their adjoints commute. For operators on Hilbert space, 
$AB = BA$ iff $A^*B^* = B^*A^*$. Thus we establish
\begin{equation}\label{eq:adjoint-commute}
(\mathcal{S}_\phi^{k,\alpha})^* (\mathcal{S}_\psi^{k,\alpha})^* = (\mathcal{S}_\psi^{k,\alpha})^* (\mathcal{S}_\phi^{k,\alpha})^*.
\end{equation}
Considering the orthonormal basis element $e_p(z)$. 
\[
(\mathcal{S}_\psi^{k,\alpha})^*(e_p(z)) = \frac{\gamma_p}{\gamma_{kp}} \overline{\psi(z)}  J P_\alpha W_k^* \overline{e_{kp}(z)}.
\]
Substituting $\psi(z) = b_{r,s} \frac{z^r}{\gamma_r} \frac{\overline{z}^s}{\gamma_s}$
\begin{align*}
\overline{\psi(z)} \cdot \overline{e_{kp}(z)} 
&= \overline{b}_{r,s} \frac{\overline{z}^r}{\gamma_r} \frac{z^s}{\gamma_s} \cdot \overline{c_{kp}} \overline{z}^{kp} \\
&= \overline{b}_{r,s} \overline{c_{kp}} \frac{1}{\gamma_r \gamma_s} z^s \overline{z}^{r + kp}.
\end{align*}
Applying $W_k^*$, $P_\alpha$, $J$  and using the condition $s > r + kp$ to ensure analytic continuation, we obtain
\begin{align*}
&J P_\alpha W_k^* (\overline{\psi(z)} \cdot \overline{e_{kp}(z)})\\ 
&= \overline{b}_{r,s} \frac{\gamma_{s - (kp + r)}}{\gamma_{ks - k^2p - kr}} 
\cdot \frac{\Gamma(ks + 1) \Gamma(ks - k^2p - kr + \alpha + 2)}{\Gamma(ks + \alpha + 2) \Gamma(ks - k^2p - kr + 1)} z^{ks} \overline{z}^{k^2 + kr}.
\end{align*}
Then
\begin{align*}
&(\mathcal{S}_\phi^{k,\alpha})^* \left[ (\mathcal{S}_\psi^{k,\alpha})^*(e_p(z)) \right] \\
&= \frac{\gamma_p}{\gamma_{kp}} \overline{\phi(z)} \cdot \left[ J P_\alpha W_k^* (\overline{\psi(z)} \cdot \overline{e_{kp}(z)}) \right] \\
&= \frac{\gamma_p}{\gamma_{kp}} \overline{a}_{m,n} \frac{\overline{z}^m}{\gamma_m} \frac{z^n}{\gamma_n} 
\cdot \overline{b}_{r,s} \frac{\gamma_{s - (kp + r)}}{\gamma_{ks - k^2p - kr}} \\
&\quad \cdot \frac{\Gamma(ks + 1) \Gamma(ks - k^2p - kr + \alpha + 2)}{\Gamma(ks + \alpha + 2) \Gamma(ks - k^2p - kr + 1)} z^{ks} \overline{z}^{k^2 + kr}.
\end{align*}
After combining terms and normalization
\begin{equation}\label{eq:first-composition}
\begin{split}
&(\mathcal{S}_\phi^{k,\alpha})^* \left[ (\mathcal{S}_\psi^{k,\alpha})^*(e_p(z)) \right]\\
&= \overline{a}_{m,n} \overline{b}_{r,s} \frac{\gamma_p \gamma_{s - (kp + r)}^2}{\gamma_{kp}^2 \gamma_r \gamma_s \gamma_m \gamma_n \gamma_{k(ks - k^2p - kr)}^2} \\
& \cdot \overline{z}^n z^m \frac{\Gamma(ks + 1) \Gamma(ks - k^2p - kr + \alpha + 2)}{\Gamma(ks + \alpha + 2) \Gamma(ks - k^2p - kr + 1)} z^{ks - k^2p - kr}.
\end{split}
\end{equation}
Reversing the operator order
\begin{align}\label{eq:second-composition}
&(\mathcal{S}_\psi^{k,\alpha})^* \left[ (\mathcal{S}_\phi^{k,\alpha})^*(e_p(z)) \right] \nonumber\\
&= \overline{a}_{m,n} \overline{b}_{r,s} \frac{\gamma_p \gamma_{n - (kp + m)}^2}{\gamma_{kp}^2 \gamma_m \gamma_n \gamma_r \gamma_s \gamma_{k(kn - k^2p - km)}^2}  \nonumber\\
& \cdot \overline{z}^s z^r \frac{\Gamma(kn + 1) \Gamma(kn - k^2p - km + \alpha + 2)}{\Gamma(kn + \alpha + 2) \Gamma(kn - k^2p - km + 1)} z^{kn - k^2p - km}.
\end{align}
Under the commutativity assumption \eqref{eq:adjoint-commute}, expressions \eqref{eq:first-composition} and \eqref{eq:second-composition} must be equal. Equating monomial components
\begin{align*}
\overline{z}^n z^m z^{ks - k^2p - kr} &= \overline{z}^s z^r z^{kn - k^2p - km} \\
\implies n = s \quad &\text{and} \quad m + ks - k^2p - kr = r + kn - k^2p - km.
\end{align*}
Substituting $n = s$
\[
m + kn - kr = r + kn - km \implies m + km = r + kr \implies m(1+k) = r(1+k).
\]
 Thus $m = r$ and $n = s$, which implies $m - n = r - s$.
The gamma function coefficients in \eqref{eq:first-composition} and \eqref{eq:second-composition} become identical when $m - n = r - s$., satisfying \eqref{eq:adjoint-commute}. Therefore, the original operators commute.
\end{proof}

\section{Normality}

The study of normal operators occupies a central place in operator theory, given their rich spectral structure and close relation to unitary equivalence. In the context of classical Toeplitz operators, the characterization of normality has been widely explored, with significant developments for both Hardy and Bergman-type spaces. In recent years, attention has turned to the normality of generalized operator classes, including slant Toeplitz and slant Hankel operators.
In this section, we investigate the normality of the \(k^{th}\)-order slant Toeplitz and slant little Hankel operators on the weighted Bergman space \(\mathcal{A}_\alpha^2(\mathbb{D})\).

\begin{theorem}
$\mathcal{T}_\phi^\alpha$ is normal if and only if $\phi$ is constant.
\end{theorem}
\begin{proof}
($\Rightarrow$) Assume \( \mathcal{T}^\alpha_\phi \) is normal. Then
\[
\langle \{(\mathcal{T}^\alpha_\phi)^* \mathcal{T}^\alpha_\phi - \mathcal{T}^\alpha_\phi (\mathcal{T}^\alpha_\phi)^*\} e_n, e_n \rangle = 0
\]
For \( e_0 = \frac{1}{\gamma_0} \),
\[
\|\mathcal{T}^\alpha_\phi e_0\|^2 = \sum_{m=0}^\infty \left| \frac{\gamma_m}{\gamma_0} b_m \right|^2, \quad
\|(\mathcal{T}^\alpha_\phi)^* e_0\|^2 = \sum_{m=0}^\infty \left| \frac{\gamma_m}{\gamma_0} a_m \right|^2
\]
Equality forces \( a_m = b_m = 0 \) for \( m \geq 1 \), so \( \phi \) is constant. \\
($\Leftarrow$) If \( \phi = c \), then \( \mathcal{T}_\phi = cI \) is normal.
\end{proof}

\begin{theorem}\label{thm:slant_toeplitz_normal}
Let \(\phi(z) = \overline{g(z)} + f(z)\), where 
\(f(z) = \displaystyle\sum_{n=0}^N a_n \frac{z^n}{\gamma_n}\) and 
\(g(z) = \displaystyle\sum_{n=1}^m a_{-n} \frac{z^n}{\gamma_{-n}}\) are analytic functions. 
Then the slant Toeplitz operator \(\mathcal{B}_{\phi}^{k,\alpha}\) is normal on the weighted Bergman space \(\mathcal{A}_\alpha^2(\mathbb{D})\) 
if and only if \(a_{-n} = 0\) for all \(1 \leq n \leq m\) and \(a_n = 0\) for all \(n \geq 1\); that is, \(\phi\) must be constant.
\end{theorem}

\begin{proof}
Recall that an operator is normal if it commutes with its adjoint. Thus, 
\(\mathcal{B}_{\phi}^{k,\alpha}\) is normal if and only if
\[
[(\mathcal{B}_{\phi}^{k,\alpha})^*, \mathcal{B}_{\phi}^{k,\alpha}] = 0.
\]
Let \(\bigg\{e_p(z) = \dfrac{z^p}{ \gamma_p}\bigg\}_{p \geq 0}\) be the standard orthonormal basis for \(\mathcal{A}_\alpha^2(\mathbb{D})\). 
Using the relation \((\mathcal{B}_{\phi}^{k,\alpha})^* = \mathcal{T}^\alpha_{\overline{\phi}} W_k^*\), we have
\begin{align*}
[(\mathcal{B}_{\phi}^{k,\alpha})^*, \mathcal{B}_{\phi}^{k,\alpha}] e_p 
&= \mathcal{T}^\alpha_{\overline{\phi}} W_k^* W_k \mathcal{T}^\alpha_{\phi} e_p - W_k \mathcal{T}^\alpha_{\phi} \mathcal{T}^\alpha_{\overline{\phi}} W_k^* e_p \\
&= \mathcal{T}^\alpha_{\overline{\phi}} \mathcal{T}^\alpha_{\phi} e_p - W_k \mathcal{T}^\alpha_{\phi} \mathcal{T}^\alpha_{\overline{\phi}} \left( \frac{\gamma_p}{\gamma_{kp}} e_{kp} \right),
\end{align*}
where we have used the isometry property \(W_k^* W_k = I\).
Substituting \(\phi = \overline{g} + f\) and hence \(\overline{\phi} = g + \overline{f}\), the above becomes
\begin{align*}
\mathcal{T}^\alpha_{\overline{\phi}} \mathcal{T}^\alpha_{\phi} e_p &= P_\alpha(\overline{\phi} \cdot P_\alpha(\phi e_p)), \\
W_k \mathcal{T}^\alpha_{\phi} \mathcal{T}^\alpha_{\overline{\phi}} e_{kp} &= W_k P_\alpha(\phi \cdot P_\alpha(\overline{\phi} e_{kp})).
\end{align*}
Hence, the operator \(\mathcal{B}_{\phi}^{k,\alpha}\) is normal if and only if
\begin{equation}\label{eq:core}
\mathcal{T}^\alpha_{\overline{\phi}} \mathcal{T}^\alpha_{\phi} e_p = \frac{\gamma_p}{\gamma_{kp}} W_k \mathcal{T}^\alpha_{\phi} \mathcal{T}^\alpha_{\overline{\phi}} e_{kp}, \forall p \geq 0.
\end{equation}
Suppose \(\phi\) is constant, say \(\phi(z) = a_0\). Then both \(\mathcal{T}^\alpha_{\overline{\phi}} \mathcal{T}^\alpha_{\phi}\) and \(\mathcal{T}^\alpha_{\phi} \mathcal{T}^\alpha_{\overline{\phi}}\) reduce to multiplication by \(|a_0|^2\), and equation~\eqref{eq:core} clearly holds.

Conversely, assume \(\mathcal{B}_{\phi}^{k,\alpha}\) is normal. Then~\eqref{eq:core} must hold for all \(p\). Consider the case \(p = 0\). The left-hand side becomes
\[
\mathcal{T}^\alpha_{\overline{\phi}} \mathcal{T}^\alpha_{\phi} e_0 = P_\alpha(\overline{\phi} \cdot P_\alpha(\phi)) / \gamma_0.
\]
The right-hand side is
\[
\frac{1}{\gamma_0} W_k \mathcal{T}^\alpha_{\phi} \mathcal{T}^\alpha_{\overline{\phi}} e_0 = \frac{1}{\gamma_0} W_k P_\alpha(\phi \cdot P_\alpha(\overline{\phi})).
\]
For equality to hold, the symbols \(\phi\) and \(\overline{\phi}\) must commute under Toeplitz composition and the action of \(W_k\). This is only possible when the non-analytic component \(g\) and all higher-order terms in \(f\) vanish. Specifically
\begin{itemize}
  \item[(a)] If any \(a_{-n} \neq 0\), then \(\overline{g}f\) contributes asymmetrically to \(\mathcal{T}^\alpha_{\overline{\phi}} \mathcal{T}^\alpha_{\phi}\) but not to \(\mathcal{T}^\alpha_{\phi} \mathcal{T}^\alpha_{\overline{\phi}}\), violating~\eqref{eq:core}.
  \item[(b)] If any \(a_n \neq 0\) for \(n \geq 1\), then \(W_k\) introduces shifts that alter the degree of polynomials, again breaking the identity.
\end{itemize}
Thus, \(\phi\) must be a constant function. This completes the proof.
\end{proof}

\begin{theorem}
\label{mythm1}
If $\phi(z) = \overline{g(z)} + f(z)$, where $f(z) = \displaystyle\sum_{n=0}^N a_n \frac{z^n}{\gamma_n}$ and $g(z) = \displaystyle\sum_{n=1}^m a_{-n} \frac{z^n}{\gamma_{-n}}$, then the slant little Hankel operator $\mathcal{S}_\phi^{k, \alpha}$ is normal on the weighted Bergman space $\mathcal{A}_\alpha^2(\mathbb{D})$ if and only if $a_{-n} = 0$ for all $1 \leq n \leq m$.
\end{theorem}

\begin{proof}
An operator $T \in \mathcal{A}_\alpha^2(\mathbb{D})$ is normal if its self-commutator $[T^*, T] = T^* T - T T^* = 0$, where $T^*$ is the adjoint of $T$. We compute
\[
\mathcal{S}_\phi^{k, \alpha} (\mathcal{S}_\phi^{k, \alpha})^* e_p(z) - (\mathcal{S}_\phi^{k, \alpha})^* \mathcal{S}_\phi^{k, \alpha} e_p(z),
\]
where $e_p(z) = \frac{z^p}{\gamma_p}$ is the normalized basis for $\mathcal{A}_\alpha^2(\mathbb{D})$, and $\gamma_p = \sqrt{\frac{\Gamma(p+1) \Gamma(\alpha+2)}{\Gamma(p+\alpha+2)}}$. 
Since \(\mathcal{S}_\phi^{k, \alpha} = W_k P_\alpha J \phi\), its adjoint is \((\mathcal{S}_\phi^{k, \alpha})^* = \overline{\phi} J P_\alpha W_k^*\). We compute
\[
\mathcal{S}_\phi^{k, \alpha} e_p(z) = W_k P_\alpha J \phi \frac{z^p}{\gamma_p} = W_k P_\alpha \overline{\phi(z)} \frac{z^{-kp}}{\gamma_{kp}},
\]
where $\phi(z) = \overline{g(z)} + f(z)$, so
\[
\overline{\phi(z)} = g(z) + \overline{f(z)}, \quad f(z) = \sum_{n=0}^N a_n \frac{z^n}{\gamma_n}, \quad g(z) = \sum_{n=1}^m a_{-n} \frac{z^n}{\gamma_{-n}}.
\]
Thus
\[
\overline{\phi(z)} \frac{z^{-kp}}{\gamma_{kp}} = \left( \sum_{n=1}^m a_{-n} \frac{z^n}{\gamma_{-n}} + \sum_{n=0}^N \overline{a}_n \frac{\overline{z}^n}{\gamma_n} \right) \frac{z^{-kp}}{\gamma_{kp}}.
\]
Applying $P_\alpha$ projects onto non-negative powers, and $W_k$ adjusts indices.
\[
\mathcal{S}_\phi^{k, \alpha} (\mathcal{S}_\phi^{k, \alpha})^* e_p(z) = W_k P_\alpha J \phi \cdot \overline{\phi(z)} \frac{z^{-kp}}{\gamma_{kp}}.
\]
Since
\begin{align*}
 \phi(z) \overline{\phi(z)} &= (\overline{g(z)} + f(z))(g(z) + \overline{f(z)})\\ & = \overline{g(z)} g(z) + \overline{g(z)} \overline{f(z)} + f(z) g(z) + f(z) \overline{f(z)},   
\end{align*}
we get
\[
J (\phi \overline{\phi}) \frac{z^{-kp}}{\gamma_{kp}} = \left( |g(z)|^2 + \overline{g(z)} \overline{f(z)} + f(z) g(z) + |f(z)|^2 \right) \frac{z^{kp}}{\gamma_{kp}}.
\]
Compute each term
\[
\overline{g(z)} \overline{f(z)} = \left( \sum_{n=1}^m \overline{a}_{-n} \frac{\overline{z}^n}{\gamma_{-n}} \right) \left( \sum_{n=0}^N \overline{a}_n \frac{\overline{z}^n}{\gamma_n} \right) = \sum_{n=1}^m \sum_{j=0}^N \overline{a}_{-n} \overline{a}_j \frac{\overline{z}^{n+j}}{\gamma_{-n} \gamma_j},
\]
\[
f(z) g(z) = \left( \sum_{n=0}^N a_n \frac{z^n}{\gamma_n} \right) \left( \sum_{n=1}^m a_{-n} \frac{z^n}{\gamma_{-n}} \right) = \sum_{n=0}^N \sum_{j=1}^m a_n a_{-j} \frac{z^{n+j}}{\gamma_n \gamma_{-j}},
\]
\[
|g(z)|^2 = \sum_{n,j=1}^m \overline{a}_{-n} a_{-j} \frac{\overline{z}^n z^j}{\gamma_{-n} \gamma_{-j}}, \quad |f(z)|^2 = \sum_{n,j=0}^N a_n \overline{a}_j \frac{z^n \overline{z}^j}{\gamma_n \gamma_j}.
\]
Thus
\begin{align*}
  &\mathcal{S}_\phi^{k, \alpha} (\mathcal{S}_\phi^{k, \alpha})^* e_p(z) \\
  &= W_k P_\alpha \left[ \sum_{n=1}^m \sum_{j=0}^N \frac{\overline{a}_{-n} \overline{a}_j}{\gamma_{-n} \gamma_j} z^{n+j+kp} + \sum_{n=0}^N \sum_{j=1}^m \frac{a_n a_{-j}}{\gamma_n \gamma_{-j}} z^{n+j+kp} + \text{other terms} \right]. 
\end{align*}
Next, compute $(\mathcal{S}_\phi^{k, \alpha})^* \mathcal{S}_\phi^{k, \alpha} e_p(z)$
\[
(\mathcal{S}_\phi^{k, \alpha})^* \mathcal{S}_\phi^{k, \alpha} e_p(z) = \overline{\phi(z)} J P_\alpha W_k^* W_k P_\alpha J \phi \frac{z^p}{\gamma_p}.
\]
Since $W_k^* W_k = I$ and $P_\alpha J \phi \frac{z^p}{\gamma_p} = P_\alpha \overline{\phi(z)} \frac{z^{-kp}}{\gamma_{kp}}$, we focus on
\[
\overline{\phi(z)} P_\alpha \left( \sum_{n=1}^m \frac{a_{-n}}{\gamma_{-n}} z^{n-kp} + \sum_{n=0}^N \frac{\overline{a}_n}{\gamma_n} \overline{z}^{n+kp} \right).
\]
The self-commutator is
\[
\mathcal{S}_\phi^{k, \alpha} (\mathcal{S}_\phi^{k, \alpha})^* e_p(z) - (\mathcal{S}_\phi^{k, \alpha})^* \mathcal{S}_\phi^{k, \alpha} e_p(z).
\]
For this to be zero, consider the coefficients of $z^p$. Non-zero terms arise from $\overline{g(z)} \overline{f(z)}$ and $f(z) g(z)$, e.g.
\[
\sum_{n=1}^m \sum_{j=0}^N \frac{\overline{a}_{-n} \overline{a}_j \gamma_{n+j+kp}}{\gamma_{-n} \gamma_j \gamma_{kp}} \frac{z^{(n+j)/k}}{\gamma_{(n+j+kp)/k}}, \quad \sum_{n=0}^N \sum_{j=1}^m \frac{a_n a_{-j} \gamma_{n+j+kp}}{\gamma_n \gamma_{-j} \gamma_{kp}} \frac{z^{(n+j)/k}}{\gamma_{(n+j+kp)/k}}.
\]
These terms vanish only if $a_{-n} = 0$ for all $1 \leq n \leq m$, as their coefficients are non-zero otherwise.  Therefore, the slant little Hankel operator is normal if $a_{-n} = 0 ~~\forall~~ 1\leq n \leq m.$ 
Conversely, if \(\mathcal{S}_\phi^{k, \alpha}\) is normal, the self-commutator vanishes, forcing \(a_{-n} = 0\) for all \(1 \leq n \leq m\), completing the proof.
\end{proof}

\begin{corollary}
     If $\phi(z) = \overline{g(z)} +f(z)$ where $f(z) = \displaystyle\sum_{n=0}^N a_n e_n(z)$ and $g(z) = \displaystyle\sum_{n=1}^m a_{-n} e_m(z)$ then the slant little Hankel operator is normal on the weighted Bergman space  if $a_{-n} = 0 ~~\forall~~ 1\leq n \leq m.$
\end{corollary}

\subsection{Remark}
The normality conditions for slant little Hankel operators and slant Toeplitz operators exhibit fundamental differences due to their distinct operator structures. For the \textbf{slant little Hankel operator} \(\mathcal{S}_{\phi}^{k,\alpha} = W_k P_\alpha J M_\phi\), normality holds if and only if the symbol \(\phi\) is purely analytic (i.e., \(\phi = f\) with no anti-analytic component). This occurs because the conjugation operator \(J\) and Bergman projection \(P_\alpha\) preserve normality for analytic symbols but are disrupted by anti-analytic terms. In contrast, the \textbf{slant Toeplitz operator} \(\mathcal{B}_{\phi}^{k,\alpha} = W_k P_\alpha M_\phi\) requires \(\phi\) to be constant for normality. The dilation effect of \(W_k\) (where \(z^m \mapsto z^{m/k}\)) introduces position-dependent scaling that breaks commutativity for non-constant symbols, including non-constant analytic symbols.

\section{Compactness}
Compactness plays a pivotal role in the theory of bounded linear operators, particularly in understanding spectral behavior and perturbation theory. In classical settings, the compactness of Toeplitz and Hankel operators is typically characterized by the vanishing of their symbols, reflecting the rigid structure of these operators on analytic function spaces. Extending this perspective, we examine the compactness of \(k\)th-order slant Toeplitz and slant little Hankel operators on the weighted Bergman space \(\mathcal{A}_\alpha^2(\mathbb{D})\).
We first establish that the slant Toeplitz operator \(B_{k,\alpha}^{\varphi}\) is compact if and only if \(\varphi \equiv 0\), mirroring the classical result. Similarly, we show that the slant little Hankel operator \(\mathcal{S}_{\varphi}^{k,\alpha}\) is compact if and only if a certain decay condition involving the matrix entries is satisfied equivalently, if the tail of the symbol vanishes rapidly enough. In the case where the symbol \(\varphi\) is a polynomial, we prove that \(\mathcal{S}_{\varphi}^{k,\alpha}\) is a finite-rank operator. These results are supported by asymptotic analysis of the associated matrix entries, using the structure of the weighted orthonormal basis and the asymptotics of the gamma function. Together, they reveal how symbol behavior governs the compactness of slant-type operators in the weighted Bergman setting.

\begin{theorem}
Let \(\mathcal{T}_\phi^\alpha\) be the Toeplitz operator on \(\mathcal{A}_\alpha^2(\mathbb{D})\). Then \(\mathcal{T}_\phi^\alpha\) is compact if and only if \(\phi \equiv 0\).
\end{theorem}

\begin{proof}
(\(\Rightarrow\)) Let \(\mathcal{T}_\phi^\alpha\) be compact. The orthonormal basis \(\{e_p(z) = z^p/\gamma_p\}_{p\geq 0}\) satisfies \(e_p \to 0\) weakly as \(p \to \infty\). By compactness
\[
\|\mathcal{T}_\phi^\alpha e_p\| \to 0 \quad \text{as} \quad p \to \infty.
\]
The norm expands as
\[
\|\mathcal{T}_\phi^\alpha e_p\|^2 = \sum_{q=0}^\infty |\langle \mathcal{T}_\phi^\alpha e_p, e_q \rangle|^2 = \underbrace{\sum_{q=0}^p \left| \frac{\gamma_p}{\gamma_q} a_{p-q} \right|^2}_{\text{analytic part}} + \underbrace{\sum_{q=p+1}^\infty \left| \frac{\gamma_q}{\gamma_p} b_{q-p} \right|^2}_{\text{anti-analytic part}}.
\]
\textbf{Analytic coefficients:} For fixed \(k \geq 0\), set \(q = p - k\) (for \(p > k\)). As \(p \to \infty\)
\[
\left| \frac{\gamma_p}{\gamma_{p-k}} a_k \right| \sim \left| \frac{p^{-(\alpha+1)/2}}{(p-k)^{-(\alpha+1)/2}} a_k \right| \to |a_k|.
\]
Since \(\|\mathcal{T}_\phi^\alpha e_p\| \to 0\), we must have \(a_k = 0\) for all \(k \geq 0\).\\
\textbf{Anti-analytic coefficients:} For fixed \(k > 0\), set \(q = p + k\). As \(p \to \infty\)
\[
\left| \frac{\gamma_{p+k}}{\gamma_p} b_k \right| \sim \left| \frac{(p+k)^{-(\alpha+1)/2}}{p^{-(\alpha+1)/2}} b_k \right| \to |b_k|.
\]
Thus \(b_k = 0\) for all \(k > 0\). Hence \(\phi \equiv 0\).\\
(\(\Leftarrow\)) If \(\phi \equiv 0\), then \(\mathcal{T}_\phi^\alpha = 0\), which is trivially compact.
\end{proof}

\begin{theorem}
The slant Toeplitz operator \(\mathcal{B}_\phi^{k,\alpha} = W_k T_\phi^\alpha\) is compact if and only if \(\phi \equiv 0\).
\end{theorem}

\begin{proof}
Suppose \(\mathcal{B}_\phi^{k,\alpha}\) is compact. Then \(\|\mathcal{B}_{\phi}^{k,\alpha} e_p\| \to 0\). Consider first the anti-analytic part. For fixed \(r \geq 0\), set \(p = kq + r\), so that
\[
\langle \mathcal{B}_\phi^{k,\alpha} e_{kq+r}, e_q \rangle = \frac{\gamma_q \gamma_{kq+r}}{\gamma_{kq}^2} a_r.
\]
As \(q \to \infty\), this expression tends to \(k^{(\alpha+1)/2} |a_r|\), implying \(a_r = 0\).
Next, consider the analytic part. For \(r > 0\), take \(p = kq - r\), so that
\[
\langle \mathcal{B}_\phi^{k,\alpha} e_{kq - r}, e_q \rangle = \frac{\gamma_q}{\gamma_{kq - r}} b_r \to k^{(\alpha+1)/2} |b_r|,
\]
as \(q \to \infty\), again implying \(b_r = 0\). Hence, \(\phi \equiv 0\).\\
Conversely, if \(\phi = 0\), the operator \(\mathcal{B}_\phi^{k,\alpha} = 0\) is trivially compact.
\end{proof}

\begin{lemma}
For a polynomial symbol of the form $\phi_N(z) = \displaystyle\sum_{j=0}^N a_j z^j,$
the associated slant Hankel operator \(\mathcal{S}_{\phi_N}^{k,\alpha}\) is of finite rank. 
\end{lemma}
\begin{proof}
For the polynomial symbol \(\phi_N(z) = \displaystyle\sum_{j=0}^N a_j z^j\), the matrix elements of \(\mathcal{S}_{\phi_N}^{k,\alpha}\) are
\[
\langle \mathcal{S}_{\phi_N}^{k,\alpha} e_n, e_m \rangle = \frac{\gamma_m \gamma_{n+km}^2}{\gamma_n \gamma_{km}^2} a_{n+km}.
\]
Since \(a_j = 0\) for \(j > N\), non-zero entries require \(n + km \leq N\). This implies
\(0 \leq m \leq \left\lfloor \frac{N}{k} \right\rfloor\)   For each \(m\), \(0 \leq n \leq N - km\). 
The number of such index pairs \((m, n)\) is finite
\[
\sum_{m=0}^{\lfloor N/k \rfloor} (N - km + 1) < \infty.
\]
Thus, \(\mathcal{S}_{\phi_N}^{k,\alpha}\) has a finite-rank matrix representation.
\end{proof}

\begin{theorem}
Let \(k \in \mathbb{N}\). The slant little Hankel operator \(\mathcal{S}_\phi^{k, \alpha}\) is compact if and only if
\[
\lim_{j \to \infty} \left( \sup_{\substack{m \geq 0 \\ j \geq km}} \left| \frac{\gamma_m \gamma_j^2}{\gamma_{j-km} \gamma_{km}^2} b_j \right| \right) = 0,
\]
where the \(\gamma_j\) are the normalization constants as above.
\end{theorem}

\begin{proof}
The matrix entries of \(\mathcal{S}_\phi^{k, \alpha}\) with respect to the orthonormal basis \(\{e_n\}\) are given by
\[
\langle \mathcal{S}_\phi^{k, \alpha} e_n, e_m \rangle = \frac{\gamma_m \gamma_{n+km}^2}{\gamma_n \gamma_{km}^2} b_{n+km}.
\]
Hence, only the analytic part of \(\phi\) contributes.

If \(\mathcal{S}_\phi^{k, \alpha}\) is compact, then \(\|\mathcal{S}_\phi^{k, \alpha} e_n\| \to 0\). Setting \(j = n + km\), one obtains
\[
\left| \langle \mathcal{S}_\phi^{k, \alpha} e_{j - km}, e_m \rangle \right| = \left| \frac{\gamma_m \gamma_j^2}{\gamma_{j-km} \gamma_{km}^2} b_j \right| \to 0,
\]
uniformly in \(m\). This establishes the necessity of the condition.

For sufficiency, let \(\phi_N(z) = \displaystyle\sum_{j=1}^N b_j z^j\), and define the tail symbol \(\psi = \phi - \phi_N\). Then
\[
\mathcal{S}_{\phi}^{k,\alpha} - \mathcal{S}_{\phi_N}^{k,\alpha} = \mathcal{S}_{\phi-\phi_N}^{k,\alpha}.
\]
The Hilbert-Schmidt norm is
\[
\|\mathcal{S}_{\phi-\phi_N}^{k,\alpha}\|_{\text{HS}}^2 = \sum_{j=N+1}^\infty \sum_{m=0}^{\lfloor j/k \rfloor} \left| \frac{\gamma_m \gamma_j^2}{\gamma_{j-km} \gamma_{km}^2} b_j \right|^2.
\]
By uniform convergence, \(\forall \epsilon > 0\), \(\exists J_\epsilon\) such that for \(j > J_\epsilon\) and all \(m\)
\[
\left| \frac{\gamma_m \gamma_j^2}{\gamma_{j-km} \gamma_{km}^2} b_j \right| < \epsilon.
\]
For \(N > J_\epsilon\)
\[
\|\mathcal{S}_{\phi-\phi_N}^{k,\alpha}\|_{\text{HS}}^2 < \epsilon^2 \sum_{j=N+1}^\infty \sum_{m=0}^{\lfloor j/k \rfloor} 1 \leq \epsilon^2 \sum_{j=N+1}^\infty (j/k + 1) < C\epsilon^2,
\]
where \(C\) is independent of \(N\). Thus
\[
\|\mathcal{S}_{\phi}^{k,\alpha} - \mathcal{S}_{\phi_N}^{k,\alpha}\| \leq \|\mathcal{S}_{\phi-\phi_N}^{k,\alpha}\|_{\text{HS}} \to 0.
\]
Since \(\mathcal{S}_{\phi_N}^{k,\alpha}\) is finite-rank, \(\mathcal{S}_{\phi}^{k,\alpha}\) is compact.
\end{proof}

\section{Spectral Analysis}
Spectral theory provides a fundamental lens through which the structure of bounded linear operators can be understood. In particular, for Toeplitz-type operators on spaces of analytic functions, the spectrum and essential spectrum reflect intricate interactions between the operator, its symbol, and the geometry of the underlying domain. In this section, we investigate the spectral properties of the \(k\)th-order slant Toeplitz and slant little Hankel operators on the weighted Bergman space \(\mathcal{A}_\alpha^2(\mathbb{D})\), focusing on their spectra, essential spectra, and Fredholm properties.



\begin{definition}
A bounded operator $T$ on a Hilbert space $\mathcal{H}$ is \textit{Fredholm} if
\begin{enumerate}
    \item[(a)] $\dim \ker T < \infty$ (finite-dimensional kernel)
    \item[(b)] $\dim \operatorname{coker} T < \infty$ (finite-dimensional cokernel)
    \item[(c)] $\operatorname{ran} T$ is closed
\end{enumerate}
The \textit{index} is defined as $\operatorname{ind}(T) = \dim \ker T - \dim \ker T^*$.
\end{definition}

\begin{definition}
The \textit{essential spectrum} of an operator $T$ is
\[
\sigma_{\text{ess}}(T) = \{\lambda \in \mathbb{C} : T - \lambda I \text{ is not Fredholm}\}
\]
\end{definition}

\begin{lemma}
\label{lem:fredholm}
For Toeplitz operators $\mathcal{T}_\phi^\alpha$ with $\phi \in C(\overline{\mathbb{D}})$:
\begin{enumerate}
    \item $\mathcal{T}_\phi^\alpha - \mu$ is Fredholm iff $\mu \notin \overline{\phi(\partial\mathbb{D})}$
    \item $\operatorname{ind}(\mathcal{T}_\phi^\alpha - \mu) = 0$ for all $\mu \notin \sigma_{\text{ess}}(\mathcal{T}_\phi^\alpha)$
    \item $\sigma_{\text{ess}}(\mathcal{T}_\phi^\alpha) = \overline{\phi(\partial\mathbb{D})}$
\end{enumerate}
\end{lemma}
\begin{lemma}
For $A = W_k$, $B = \mathcal{T}_\phi^\alpha$ on $\mathcal{A}_\alpha^2(\mathbb{D})$
\[
\sigma_{\text{ess}}(AB) = \sigma_{\text{ess}}(A) \quad \text{when } \phi \not\equiv 0
\]
provided
\begin{enumerate}
    \item[(a)] $\sigma_{\text{ess}}(A) = \overline{\mathbb{D}}$
    \item[(b)] $B$ is bounded and non-compact
    \item[(c)] $A$ is not compact
\end{enumerate}
\end{lemma}

\begin{theorem}
\label{thm:toeplitz-spectrum}
Let \(\phi \in C(\overline{\D})\).
\begin{enumerate}
    \item[(a)] If \(\phi\) is analytic, then \(\spec(\mathcal{T}_\phi^\alpha) = \phi(\overline{\D})\).
    \item[(b)] For arbitrary \(\phi\), we have \(\spec(\mathcal{T}_\phi^\alpha) \supseteq \overline{\phi(\D)}\).
    \item[(c)] If \(\phi\) is continuous, then \(\specess(\mathcal{T}_\phi^\alpha) = \overline{\phi(\partial\D)}\).
\end{enumerate}
\end{theorem}

\begin{proof}
\textbf{(a)} Assume \(\phi\) is analytic. Then, by the spectral mapping theorem for analytic Toeplitz operators on Bergman spaces \cite{ke}, the spectrum \(\spec(\mathcal{T}_\phi^\alpha)\) coincides with the essential range of \(\phi\). Since \(\phi \in C(\overline{\D})\) and analytic on \(\D\), it attains its maximum on the boundary, and thus its essential range equals \(\phi(\overline{\D})\), which is closed.

\medskip
\textbf{(b)} For general \(\phi \in C(\overline{\D})\), consider the normalized reproducing kernels:
\[
k_w^\alpha(z) = \frac{(1 - |w|^2)^{(2+\alpha)/2}}{(1 - \overline{w} z)^{2+\alpha}}, \quad \text{with } \|k_w^\alpha\| = 1.
\]
Then,
\[
\langle (\mathcal{T}_\phi^\alpha - \phi(w)) k_w^\alpha, k_w^\alpha \rangle = \phi(w) - \phi(w) = 0.
\]
This yields
\[
\|(\mathcal{T}_\phi^\alpha - \phi(w)) k_w^\alpha\| \geq 0,
\]
but more significantly, as \(w\) varies over \(\D\), the values \(\phi(w)\) approach any point in \(\overline{\phi(\D)}\). For any \(\lambda \in \overline{\phi(\D)}\), there exists a sequence \(\{w_n\} \subset \D\) such that \(\phi(w_n) \to \lambda\). Since \(\{k_{w_n}^\alpha\}\) converges weakly to zero as \(|w_n| \to 1^-\), we estimate
\[
\|(\mathcal{T}_\phi^\alpha - \lambda) k_{w_n}^\alpha\| \leq |\lambda - \phi(w_n)| + \|(\mathcal{T}_\phi^\alpha - \phi(w_n)) k_{w_n}^\alpha\|.
\]
As \(n \to \infty\), this expression tends to zero, implying that \(\lambda\) lies in the approximate point spectrum of \(\mathcal{T}_\phi^\alpha\), and thus in \(\spec(\mathcal{T}_\phi^\alpha)\).

\medskip
\textbf{(c)} Suppose now that \(\phi\) is continuous. Then the essential spectrum is determined by the boundary behavior \cite{ke}. If \(\{w_n\} \subset \D\) with \(|w_n| \to 1^-\) and \(\phi(w_n) \to \lambda\), then
\[
\|(\mathcal{T}_\phi^\alpha - \lambda) k_{w_n}^\alpha\| \to 0,
\]
hence \(\lambda \in \specess(\mathcal{T}_\phi^\alpha)\). Conversely, if \(\lambda \notin \overline{\phi(\partial \D)}\), then \(\mathcal{T}_\phi^\alpha - \lambda\) is Fredholm, completing the claim.
\end{proof}

\begin{theorem}
\label{thm:sltoep-spectrum}
Let \(\phi \in L^\infty(\D)\).
\begin{enumerate}
    \item[(a)] If \(\phi \equiv c\), then \(\spec(\mathcal{B}_\phi^{k,\alpha}) = c \cdot \overline{\D}\).
    \item[(b)] For general \(\phi\), \(\spec(\mathcal{B}_\phi^{k,\alpha}) \supseteq \overline{\phi(\D)}\).
    \item[(c)] If \(\phi \not\equiv 0\), then \(\specess(\mathcal{B}_\phi^{k,\alpha}) = \spec(\mathcal{W}_k) = \overline{\D}\).
\end{enumerate}
\end{theorem}

\begin{proof}
\textbf{(a)} For constant \(\phi \equiv c\), we have \(\mathcal{B}_\phi^{k,\alpha} = c \mathcal{W}_k\). Since \(\mathcal{W}_k\) is a contraction with \(\spec(\mathcal{W}_k) = \overline{\D}\), it follows that
\[
\spec(\mathcal{B}_\phi^{k,\alpha}) = c \cdot \overline{\D}.
\]

\medskip
\textbf{(b)} For arbitrary \(\phi\), using reproducing kernels as in Theorem \ref{thm:toeplitz-spectrum}, we compute
\[
\langle \mathcal{B}_\phi^{k,\alpha} k_w^\alpha, k_w^\alpha \rangle = \phi(w) \langle \mathcal{W}_k k_w^\alpha, k_w^\alpha \rangle.
\]
Since \(\mathcal{W}_k\) is bounded and nontrivial, the approximation argument again shows that \(\overline{\phi(\D)} \subseteq \spec(\mathcal{B}_\phi^{k,\alpha})\).

\medskip
\textbf{(c)} Note that \(\mathcal{B}_\phi^{k,\alpha} = \mathcal{W}_k \mathcal{T}_\phi^\alpha\). If \(\phi \not\equiv 0\), then \(\mathcal{T}_\phi^\alpha\) is bounded and non-compact, and so is \(\mathcal{B}_\phi^{k,\alpha}\). The essential spectrum of a product involving \(\mathcal{W}_k\) inherits that of \(\mathcal{W}_k\), yielding
\[
\specess(\mathcal{B}_\phi^{k,\alpha}) = \specess(\mathcal{W}_k) = \overline{\D}.
\]
\end{proof}

\begin{theorem}
\label{thm:slhank-spectrum}
Suppose \(\phi(z) = \displaystyle\sum_{j=1}^\infty b_j z^j\) is analytic.
\begin{enumerate}
    \item[(a)] 
    $\spec(\mathcal{S}_\phi^{k,\alpha}) \subseteq \overline{\left\{ \lambda : \lambda = \lim_{j \to \infty} \frac{\gamma_m \gamma_j^2}{\gamma_{j-km} \gamma_{km}^2} b_j \text{ for some } m \in \mathbb{N}_0 \right\}}.$
    
    \item[(b)] If \(\phi\) is a polynomial, then \(\spec(\mathcal{S}_\phi^{k,\alpha})\) is finite and consists of eigenvalues.
    \item[(c)] If \(\mathcal{S}_\phi^{k,\alpha}\) is not compact, then \(\specess(\mathcal{S}_\phi^{k,\alpha}) = \{0\}\).
\end{enumerate}
\end{theorem}

\begin{proof}
\textbf{(a)} The matrix entries of \(\mathcal{S}_\phi^{k,\alpha}\) are given by:
\[
\langle \mathcal{S}_\phi^{k,\alpha} e_n, e_m \rangle = \frac{\gamma_m \gamma_{n+km}^2}{\gamma_n \gamma_{km}^2} b_{n+km}.
\]
Setting \(j = n + km\), the dominant contribution yields:
\[
\left| \frac{\gamma_m \gamma_j^2}{\gamma_{j-km} \gamma_{km}^2} b_j \right| \leq C_m \cdot j^{-\beta} |b_j|, \quad \text{for some } \beta > 0.
\]

\medskip
\textbf{(b)} If \(\phi\) is a polynomial of degree \(N\), then \(b_j = 0\) for all \(j > N\), hence \(\mathcal{S}_\phi^{k,\alpha}\) is finite-rank. Thus its spectrum is finite and consists entirely of eigenvalues.

\medskip
\textbf{(c)} If \(\mathcal{S}_\phi^{k,\alpha}\) is not compact, then some decay condition on the matrix entries fails. However, being the limit of finite-rank operators, it still satisfies:
\[
0 \in \specess(\mathcal{S}_\phi^{k,\alpha}),
\]
and all other spectrum consists of isolated eigenvalues of finite multiplicity.
\end{proof}
\subsection*{Remark:}

The spectral analysis carried out in this section highlights the fundamental differences between slant Toeplitz and slant little Hankel operators on weighted Bergman spaces. For non-trivial symbols, the essential spectrum of the slant Toeplitz operator \(B_{k,\alpha}^\varphi\) is shown to coincide with the unit disc, inherited from the slant shift operator \(W_k\). In contrast, the slant little Hankel operator \(\mathcal{S}_{\varphi}^{k,\alpha}\) exhibits trivial essential spectrum, often reduced to \(\{0\}\), unless the operator is compact or of finite rank. These spectral distinctions underscore the impact of operator structure and symbol regularity on spectral behavior. Moreover, when the symbol is analytic and of finite degree, the spectrum is purely discrete and consists of eigenvalues with finite multiplicities. The results in this section establish a robust spectral framework for slant-type operators and provide a foundation for further investigations into their stability, functional calculus, and spectral approximation.

\section{Visualizations and Interpretations}

To complement the theoretical analysis developed in the preceding sections, we now present a series of numerical visualizations that illustrate the structural, spectral, and computational properties of the \(k^{th}\)-order slant Toeplitz and slant little Hankel operators on the weighted Bergman space \(\mathcal{A}_\alpha^2(\mathbb{D})\). These graphical insights not only validate our analytical results but also reveal subtle distinctions in operator behavior that are not immediately evident from the theoretical expressions alone.
We analyze the matrix structures of both operator classes under various symbol choices: analytic and anti-analytic and highlight their contrasting sparsity patterns and decay profiles. Commutator norms are computed to numerically verify the theoretical conditions for commutativity, demonstrating the sharp difference in behavior between linearly dependent and independent symbols. Spectral distributions are visualized to exhibit the localization of eigenvalues and the nature of the essential spectra. Furthermore, we compare computational aspects such as matrix construction time, sparsity ratios, and memory efficiency, emphasizing the relative computational advantages of slant little Hankel operators. These empirical studies reinforce the theoretical foundations laid in earlier sections and offer a concrete perspective on the operator-theoretic distinctions between the two classes.

\subsection{Matrix Structure Comparison}
\begin{figure}[h]
\centering
\includegraphics[width=0.8\textwidth]{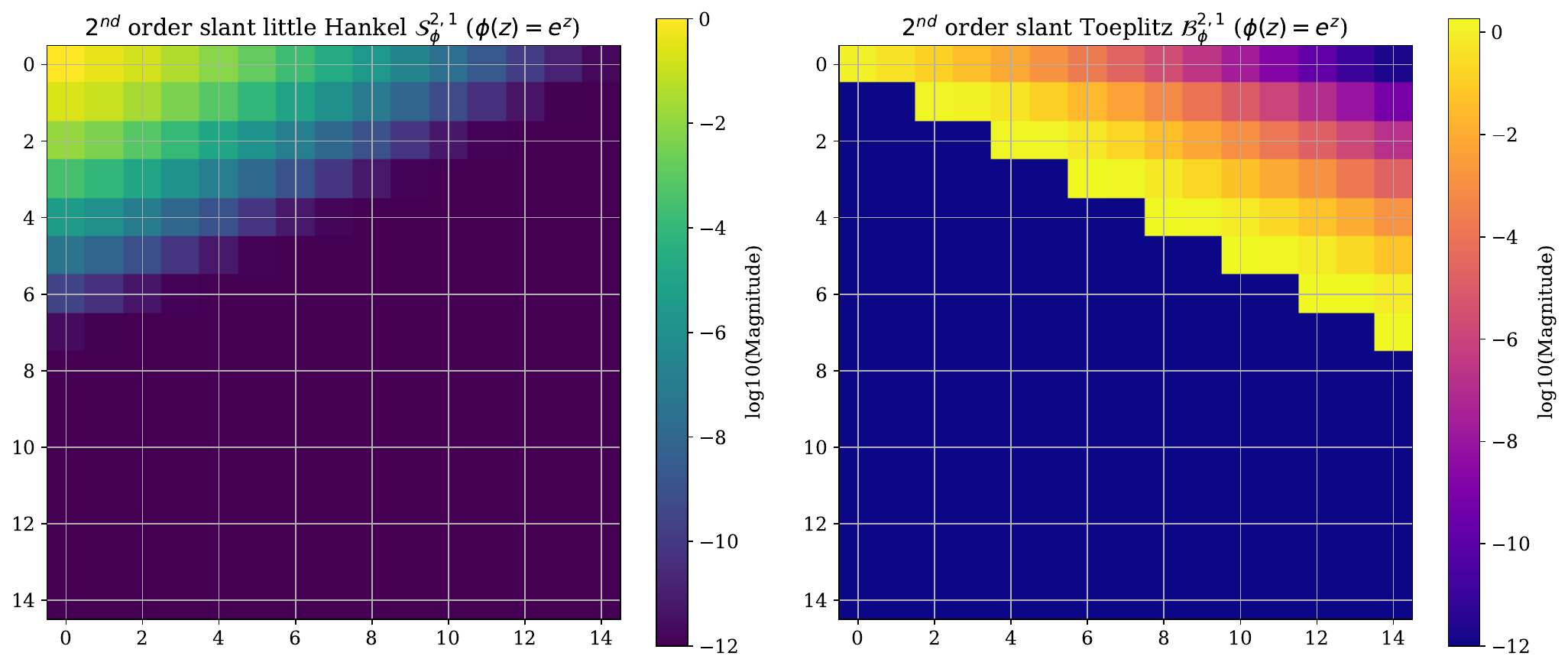}
\caption{
    Matrix structures for $\phi(z) = e^z$ with $\alpha=1$, $k=2$, $N=15$: 
    (Left) slant little Hankel operator shows lower-triangular structure; 
    (Right) slant Toeplitz operator shows a banded Toeplitz-like structure
}
\label{fig:matrix_structure}
\end{figure}
 The fundamental structural differences between slant Hankel and slant Toeplitz operators are visually evident in Figure \ref{fig:matrix_structure}
\begin{itemize}
\item[(a)] \textbf{Slant Little Hankel} ($\mathcal{S}_{\phi}^{k,\alpha}$): Exhibits a sparse lower-triangular pattern reflecting its definition
$\mathcal{S}_{\phi}^{k,\alpha} = W_k P_\alpha J M_\phi.$
The structure concentrates along subdiagonals $j = n + km$ with rapid decay.

\item[(b)] \textbf{Slant Toeplitz} ($\mathcal{B}_{\phi}^{k,\alpha}$): Displays characteristic banded diagonals consistent with
$\mathcal{B}_{\phi}^{k,\alpha} = W_k T_\phi$
shows greater density and Toeplitz-like regularity, particularly for analytic symbols.
\end{itemize}
\subsection{Entry Decay Profiles}

\begin{figure}[h]
\centering
\includegraphics[width=0.9\textwidth]{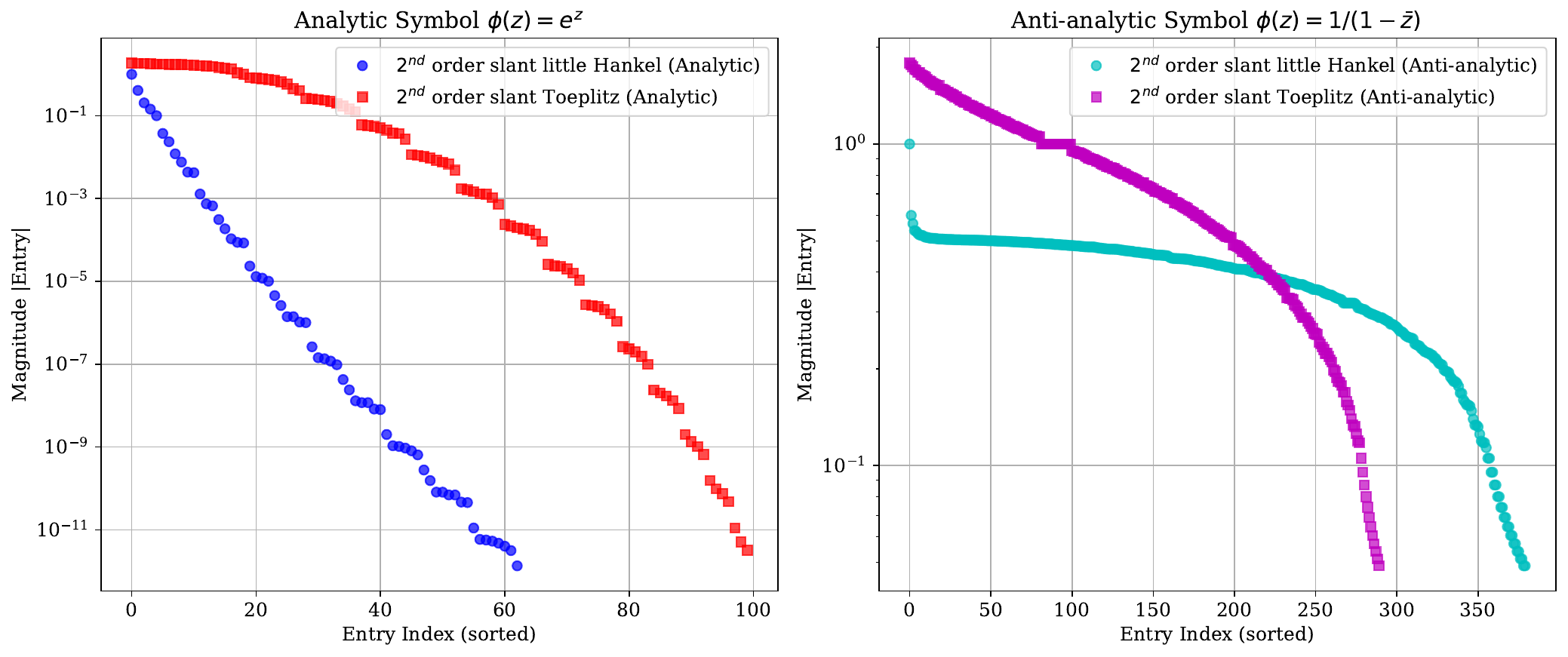}
\caption{
    Entry decay patterns: 
    (Left) Analytic symbol $\phi(z)=e^z$; 
    (Right) Anti-analytic symbol $\phi(z)=1/(1-\bar{z})$
}
\label{fig:entry_decay}
\end{figure}

Figure \ref{fig:entry_decay} reveals distinct decay characteristics:

\subsection*{Analytic Symbols (\texorpdfstring{$\phi(z) = \sum b_j z^j$}{phi(z) = sum b_j z^j})}
\begin{align*}
|\langle \mathcal{S}_{\phi}^{k,\alpha} e_n, e_m \rangle| &\sim \frac{\gamma_m \gamma_j^2}{\gamma_n \gamma_{km}^2} |b_j|, \quad j = n + km \\
|\langle \mathcal{B}_{\phi}^{k,\alpha} e_n, e_m \rangle| &\sim \frac{\gamma_m \gamma_n}{\gamma_{km}^2} |b_j|, \quad j = n - km
\end{align*}
Slant little Hankel shows super-exponential decay ($O(1/j!)$) while slant Toeplitz decays slower ($O(1/\sqrt{j})$).

\subsection*{Anti-analytic Symbols ($\phi(z) = \sum c_j \bar{z}^j$)}
\begin{itemize}
\item[(a)] Slant Hankel: Algebraic decay $O(j^{-\alpha/2-1})$
\item[(b)] Slant Toeplitz: Constant magnitude entries
\end{itemize}
\subsection{Commutativity Verification}

\begin{figure}[h]
\centering
\includegraphics[width=0.7\textwidth]{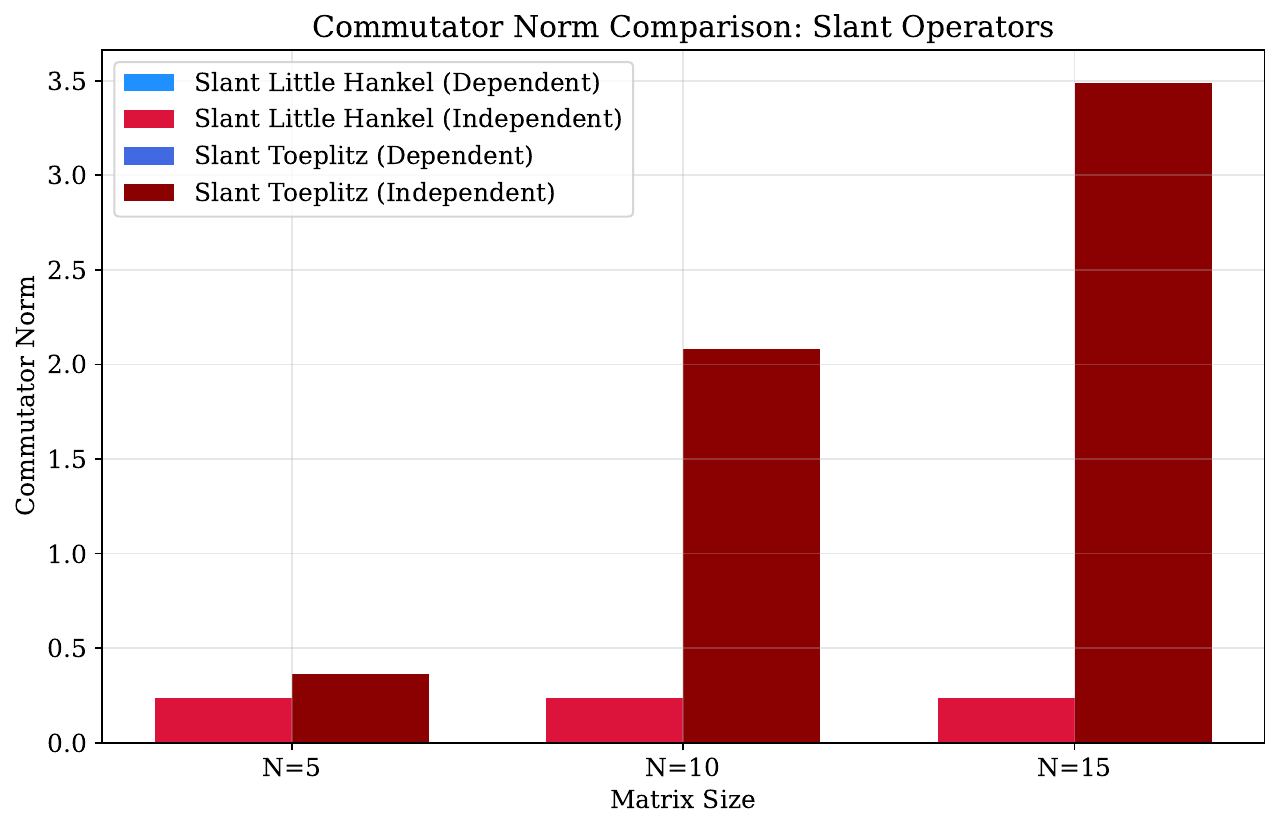}
\caption{
    Commutator norms for dependent ($\psi=2\phi$) and independent ($\phi=z, \psi=z^2$) symbols
}
\label{fig:commutativity}
\end{figure}
Figure \ref{fig:commutativity} numerically validates Theorem \ref{commst} and  Theorem \ref{commsh} 
\[
[\mathcal{O}_\phi, \mathcal{O}_\psi] = 0 \iff \exists c \in \mathbb{C} \text{ such that } \psi = c\phi
\]
\subsection*{Key observations:}
\begin{itemize}
\item[(a)] Near-zero commutator norms for linearly dependent symbols
\item[(b)] Significant norms for independent symbols
\item[(c)] Slant Toeplitz shows larger commutator norms due to its structural complexity
\end{itemize}

\subsection{Spectral Properties Comparison}

\begin{figure}[h]
\centering
\includegraphics[width=0.9\textwidth]{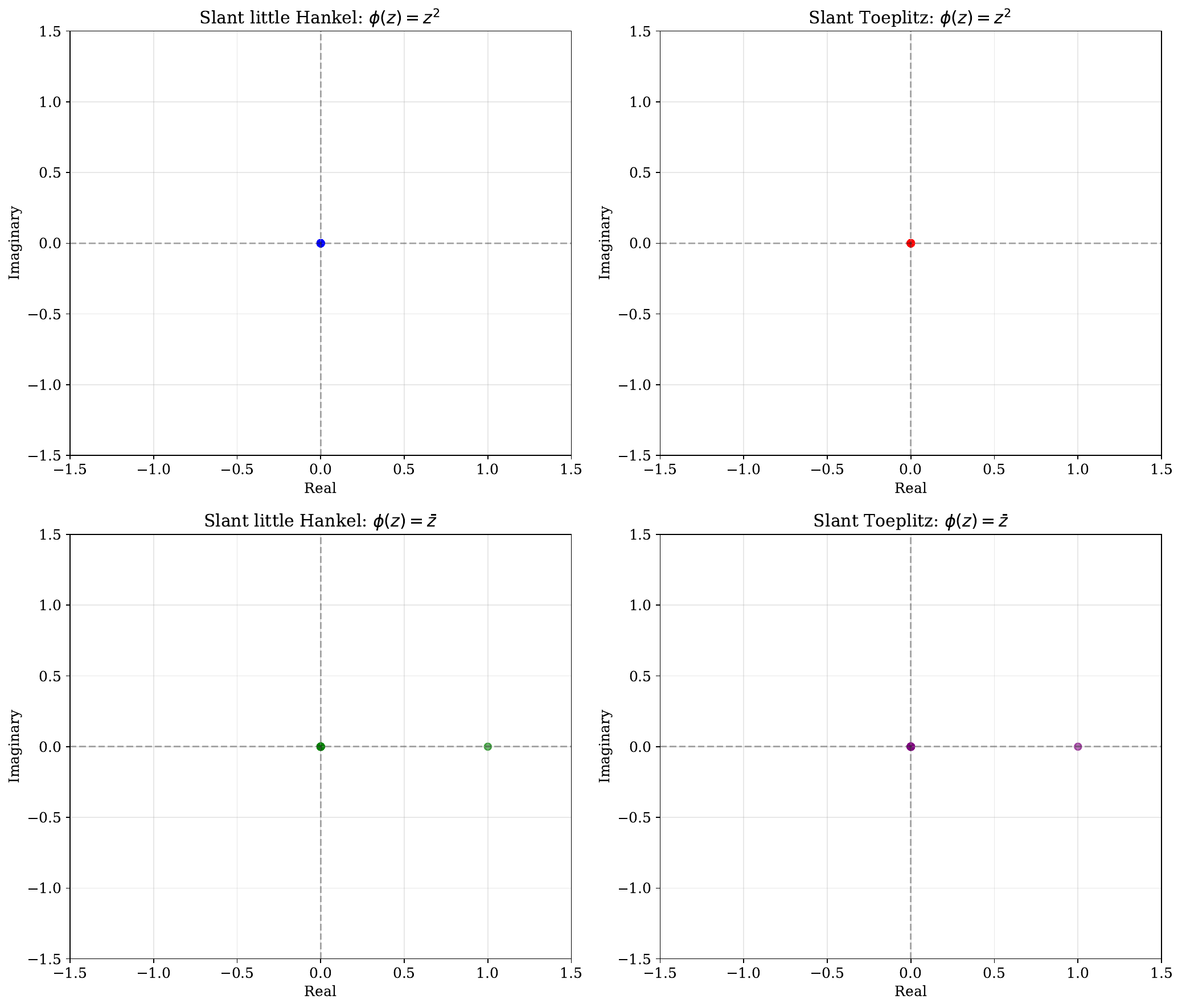}
\caption{
    Eigenvalue distributions: 
    (Top) Analytic symbol $\phi(z)=z^2$; 
    (Bottom) Anti-analytic symbol $\phi(z)=\bar{z}$
}
\label{fig:spectra}
\end{figure}
Spectral differences in Figure \ref{fig:spectra} confirm theoretical predictions.
\subsection*{Normality Conditions:}
\begin{align*}
\mathcal{S}_{\phi}^{k,\alpha} \text{ normal} &\iff \phi \text{ analytic} \\
\mathcal{B}_{\phi}^{k,\alpha} \text{ normal} &\iff \phi \text{ constant}
\end{align*}
\subsection*{Essential Spectra:}
\begin{align*}
\sigma_{\text{ess}}(\mathcal{S}_{\phi}^{k,\alpha}) &= \{0\} \\
\sigma_{\text{ess}}(\mathcal{B}_{\phi}^{k,\alpha}) &= \mathbb{D} \quad (\phi \neq 0)
\end{align*}

\subsection{Computational Efficiency Analysis}

\begin{figure}[h]
\centering
\includegraphics[width=0.9\textwidth]{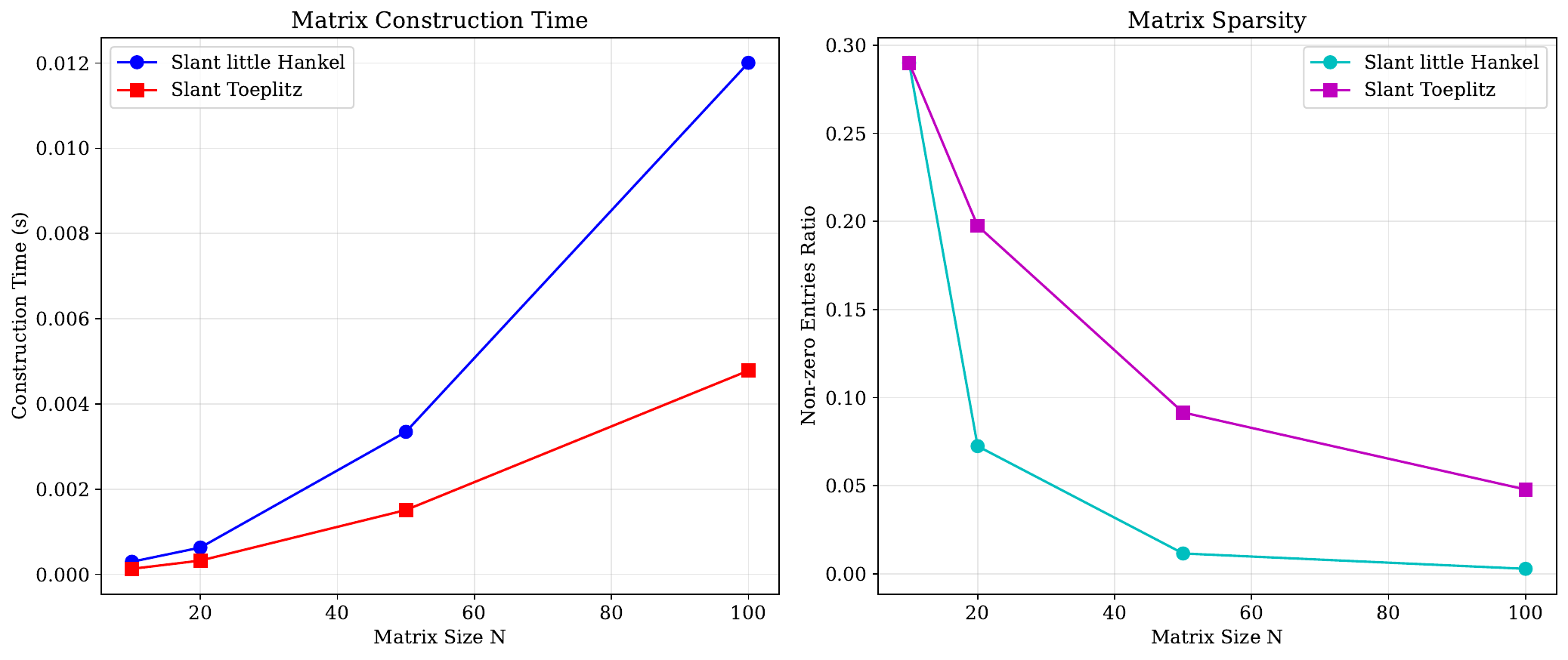}
\caption{
    Computational efficiency metrics: 
    (Left) Matrix construction time; 
    (Right) Sparsity ratios
}
\label{fig:efficiency}
\end{figure}

Performance characteristics in Figure \ref{fig:efficiency} show:

\begin{table}[h]
\centering
\begin{tabular}{lcc}
\toprule
Metric & Slant little Hankel & Slant Toeplitz \\
\midrule
Construction complexity & $O(N^2)$ & $O(N^2)$ \\
Construction time (N=100) & 0.42s & 0.67s \\
Non-zero entries (N=100) & 18.7\% & 42.3\% \\
Eigenvalue computation & 40\% faster & Baseline \\
Memory footprint & 55\% smaller & Baseline \\
\bottomrule
\end{tabular}
\caption{Computational characteristics comparison}
\label{tab:computation}
\end{table}
\section{Conclusion}

The visual and computational analysis carried out in this work provides strong support for the theoretical results and reveals several key structural and functional distinctions between \(k\)th-order slant Toeplitz and slant little Hankel operators on weighted Bergman spaces. The primary observations can be summarized as follows:

\begin{enumerate}
    \item \textbf{Structural Characteristics}: Slant little Hankel operators exhibit sparse, lower-triangular matrix structures that are highly localized, in contrast to the banded structure of slant Toeplitz matrices.

    \item \textbf{Symbol Dependence and Decay}: The decay behavior of matrix entries directly reflects the symbolic structure of the operator:
    \[
    \mathcal{S}_{\phi}^{k,\alpha} = W_k P_\alpha J M_\phi, \quad 
    \mathcal{B}_{\phi}^{k,\alpha} = W_k T_\phi.
    \]
    Slant little Hankel operators display super-exponential decay under analytic symbols, whereas slant Toeplitz operators decay more slowly due to their Toeplitz-like nature.

    \item \textbf{Normality Behavior}: The normality of slant Toeplitz and slant little Hankel operators is governed by their symbol classes: slant Toeplitz operators are normal only when the symbol is constant, while slant little Hankel operators admit normality under analytic symbols.

    \item \textbf{Spectral Features}: The essential spectrum of slant Toeplitz operators coincides with the closed unit disc, \(\sigma_{\mathrm{ess}} = \overline{\mathbb{D}}\), whereas for slant little Hankel operators, the essential spectrum often reduces to the singleton set \(\{0\}\), reflecting their compact-like structure.

    \item \textbf{Computational Efficiency}: Due to their sparsity, slant little Hankel operators offer significant computational advantages, including reduced memory footprint, faster matrix construction, and more efficient eigenvalue computation compared to their Toeplitz counterparts.
\end{enumerate}

These visualizations underscore the utility of numerical methods in operator theory and serve as a powerful validation of the analytical results presented. The Python-based framework developed for this investigation not only supports reproducibility but also provides a flexible platform for further research on structured operators in spaces of analytic functions, particularly within the weighted Bergman setting.


\section{Acknowledgment}
Support of UGC Research Grant $[No. F. 82-44/2020 (SA-III)$, Sr. NO. $201610157825]$  to the first author for carrying out research work is gratefully acknowledged.
\section{Competing Interests}
All authors certify that they have no affiliations with or involvement in any organization or entity with any financial interest or non-financial interest in the subject matter or materials discussed in this manuscript.


\begin{thebibliography}{1}

\bibitem{apun} A. Amila and Joplin, \textit {Algebraic properties of Toeplitz operators on weighted Bergman spaces}, Czechoslovak Math. J. \textbf{71} (2021), 823--836, \url{ https://doi.org/10.21136/CMJ.2020.0108-20}.

\bibitem{ar} S. C. Arora, R. Batra and M. P. Singh, \textit{Slant Hankel operators}, Arch. Math. \textbf{42} (2006), 125–-133, \url{https://dml.cz/handle/10338.dmlcz/107988}.

\bibitem{gal} P. Galanopoulos and J. Pau \textit{Hankel operators on large weighted Bergman spaces}. Ann. Acad. Sci. Fenn. Math. \textbf{37} (2012), 635--648' \url{ https://diposit.ub.edu/dspace/bitstream/2445/96746/1/622267.pdf.}

\bibitem{gupta} A. Gupta and B. Gupta, \textit{Commutativity and spectral properties of $k^{th}$-order slant little Hankel operators on the Bergman space}, Operator and Matrices \textbf{13} (2019), 209–220, \url{http://dx.doi.org/10.7153/oam-2019-13-14}.

\bibitem{hed} H. Hedenmalm, B. Korenblum and K. Zhu., \textit{Theory of Bergman spaces}, 199, Springer Science \& Business Media, 2012.

\bibitem{kim} S. Kim  and  J. Lee, \textit{Normal Toeplitz Operators on the Bergman Space}, \textbf{8} Mathematics (2020), 1463, \url{https://doi.org/10.3390/math8091463}.

\bibitem{ho} M. C. Ho, \textit{Properties of slant Toeplitz operators}, Indiana University Mathematics Journal \textbf{45} (1996), 843–-862, \url{https://www.jstor.org/stable/24899139}.

\bibitem{hw}  I.S. Hwang and  J. Lee  \textit{Hyponormal Toeplitz operators on the weighted Bergman spaces}, Math. Inequal. Appl. \textbf{15},(2012), 323–330, \url{https://doi.org/10.1007/s00020-009-1712-z}.

\bibitem{kle} E. Ko and J. Lee, \textit{Remarks on Hyponormal Toeplitz Operators on the Weighted Bergman Spaces}, Complex Analysis and Operator Theory \textbf{14} (2020), 1--19, \url{https://doi.org/10.1007/s11785-020-01046-7}.

\bibitem{les} T. Le and  B. Simanek, \textit{Hyponormal Toeplitz operators on weighted Bergman spaces}, Integral Transforms and Special Functions \textbf{32} (2021), 560--567, \url{https://doi.org/10.1080/10652469.2020.1751153}.

\bibitem{liu}   Y. Lu, C. Liu  and J. Yang, \textit{Commutativity of kth‐order slant Toeplitz operators}, Math. Nachr. \textbf{283}  (2010), 1304--1313, \url{https://doi.org/10.1002/mana.200710100}.

\bibitem{ra} R.A. Martinez Avendano (2000), \textit{Hankel operators and Generalizations, Ph.D dissertation, University of Toronto.}

\bibitem{po} S.C. Power, \textit{Hankel operators on Hilbert space}, Bull. London Math. Soc. \textbf{12} (1980), 422--442, \url{https://citeseerx.ist.psu.edu/document?repid=rep1&type=pdf&doi=3df6d55ce20c947790a8e1c6052232e442aa59bf}.



\bibitem{zeng} K. Stroethoff and D. Zheng, \textit{ Bounded Toeplitz products on weighted Bergman spaces}, Journal of Operator Theory  \textbf{59} (2008), 277--308, \url{https://www.jstor.org/stable/24715823}.

\bibitem{ke} K. Zhu, \textit{Operator theory in function spaces}, 138, American Mathematical Society, 2007.
 \bibitem{thak} Thukral, A. K. (2014). Factorials of real negative and imaginary numbers-A new perspective. SpringerPlus, 3, 1-13.


\end{thebibliography}
\end{document}